\documentclass[12pt]{amsart}

\usepackage{mathtools}
\mathtoolsset{showonlyrefs=true}
\usepackage[hmargin=0.8in,height=8.6in]{geometry}
\usepackage{amssymb,amsthm, times}
\usepackage{delarray,verbatim}
\usepackage{ifpdf}
\ifpdf
\usepackage[pdftex]{graphicx}
 \else
\usepackage[dvips]{graphicx}
 \fi

\usepackage{bm}

\usepackage{mathtools}
\mathtoolsset{showonlyrefs=true}

\usepackage{natbib}
 \bibpunct[, ]{[}{]}{,}{n}{}{,}%

\usepackage[colorlinks=true,breaklinks=true,bookmarks=true,urlcolor=blue,
     citecolor=blue,linkcolor=blue,bookmarksopen=false,draft=false]{hyperref}

\def\EMAIL#1{\href{mailto:#1}{#1}}
\def\URL#1{\href{#1}{#1}}         

\begin{document}

\title{{On singular control for L\'evy processes}}

	\thanks{This version: \today.   }

\author[K. Noba]{Kei Noba$^*$}
\thanks{$*$\, School of Statistical Thinking, The Institute of Statistical Mathematics, 10-3 Midori-cho, Tachikawa-shi, Tokyo 190-8562, Japan, \EMAIL{knoba@ism.ac.jp}, \URL{https://sites.google.com/view/knoba}}
\author[K. Yamazaki]{Kazutoshi Yamazaki$^\dagger$}
\thanks{$\dagger$\, School of 
Mathematics and Physics, The University of Queensland, St Lucia, Brisbane, QLD 4072, Australia, \EMAIL{k.yamazaki@uq.edu.au}, \URL{https://sites.google.com/site/kyamazak/ }}
\date{}

\newcommand{\absol}[1]{\left| #1 \right|} 
\newcommand{\norm}[1]{\left\| #1 \right\|} 
\newcommand{\rbra}[1]{\!\left( #1 \right)} 
\newcommand{\cbra}[1]{\!\left\{ #1 \right\}} 
\newcommand{\sbra}[1]{\!\left[ #1 \right]} 
\newcommand{\abra}[1]{\!\left\langle #1 \right\rangle} 
\newcommand{\floor}[1]{\!\left\lfloor #1 \right\rfloor \!} 
\newcommand{\ceil}[1]{\!\left\lceil #1 \right\rceil} 

\newcommand{\cA}{\ensuremath{\mathcal{A}}}

\newcommand{\bE}{\ensuremath{\mathbb{E}}}
\newcommand{\bP}{\ensuremath{\mathbb{P}}}
\newcommand{\bR}{\ensuremath{\mathbb{R}}}

\newcommand{\bN}{\ensuremath{\mathbb{N}}}

\newcommand {\ME}{\mathbb{E}^{x}}
\newcommand {\MEX}{\mathbb{E}^{X_{T_1}}}
\newcommand {\MEt}{\mathbb{E}^{x}_{t}}
\newcommand {\uEt}{\tilde{\mathbbm{E}}}
\newcommand {\ud}{{\rm d}}
\newcommand {\bracks}[1]{\left[#1\right]}
\newcommand {\parens}[1]{\left(#1\right)}
\newcommand {\braces}[1]{\left\{#1\right\}}
\newcommand {\I}[1]{\mathbbm{1}_{\{#1\}}}
\newcommand {\myBox}{\hspace{\stretch{1}}$\diamondsuit$}
\renewcommand{\S}{\mathcal{S}}
\renewcommand{\theequation}{\thesection. \arabic{equation}}

\newcommand{\un}[1]{\underline{#1}}

\newcommand{\lapinv}{\Phi(q)}
\newcommand {\B}{\mathcal{B}}

\newcommand {\ph}{\hat{p}}
\newcommand {\ah}{\hat{C}}
\newcommand {\pcheck}{\tilde{p}}
\newcommand {\acheck}{\tilde{C}}

\newcommand{\cL}{\ensuremath{\mathcal{L}}}
\newcommand{\cN}{\ensuremath{\mathcal{N}}}
\newcommand{\cB}{\ensuremath{\mathcal{B}}}
\newcommand{\cF}{\ensuremath{\mathcal{F}}}

\newcommand {\R}{\mathbb{R}}
\newcommand {\D}{e^{-C\tau}}
\newcommand {\F}{\mathcal{F}}
\newcommand {\A}{\mathcal{A}}
\newcommand {\N}{\mathbb{N}}
\newcommand {\p}{\mathbb{P}}
\newcommand {\G}{\mathcal{G}}
\newcommand {\E}{\mathbb{E}}
\newcommand {\T}{\blue{\mathcal{R}}}
\newcommand {\EE}{\mathcal{E}}
\newcommand {\QV}{\langle M, M\rangle}
\newcommand {\QVN}{\langle N, N\rangle}
\newcommand {\CV}{\langle M, N\rangle}
\newcommand {\SQ}{[X]}
\newcommand {\M}{\mathcal{M}}
\newcommand {\X}{M=(M_t)_{t\in \R_+}}
\newcommand {\MM}{\mathcal{M}_2}
\newcommand {\MMc}{\mathcal{M}_2^c}
\newcommand {\LO}{\mathcal{M}^{loc}}
\newcommand {\Lc}{\mathcal{M}^{c, loc}}
\newcommand {\IM}{I_t^M(X)}
\newcommand {\IN}{I_t^N(Y)}
\newcommand{\qvIM}{\langle I^M(X), I^M(X)\rangle}
\newcommand{\qvIN}{\langle I^N(Y), I^N(Y)\rangle}
\newcommand {\LIM}{\lim_{n\rightarrow \infty}}
\newcommand{\Pstar}{\mathcal{P}^*}
\newcommand{\diff}{{\rm d}}
\newcommand{\disc}{e^{-C T}}
\newcommand{\nn}{\nonumber}
\newcommand{\U}{U^C f}
\newcommand{\lev}{L\'{e}vy }
\newcommand{\zq}{e^{-\Phi(q) y}}
\newcommand{\red}{\textcolor[rgb]{1.00,0.00,0.00}}
\newcommand{\blue}{\textcolor[rgb]{0.00,0.00,1.00}}
\newcommand{\green}{\textcolor[rgb]{0.00,0.50,0.50}}
\newcommand{\px}{\mathbb{P}_x}
\newcommand{\e}{\mathbb{E}}
\newcommand{\ex}{\mathbb{E}_x}
\newcommand{\eq}{\mathbb{E}^{\mathbb{Q}}}
\newcommand{\eqx}{\mathbb{E}^{\mathbb{Q}}_x}
\newcommand{\pf}{\mathbb{P}^\uparrow}
\newcommand{\pfx}{\mathbb{P}_{x}^\uparrow}
\newcommand{\ix}{\underline{X}}
\newcommand{\W}{\mathbb{W}}

\renewcommand{\labelenumi}{(\roman{enumi})}

\newcommand{\EQ}{\rm{e}_q}

\def\TheoremsNumberedThrough{%
\theoremstyle{TH}%
\newtheorem{theorem}{Theorem}
\newtheorem{lemma}{Lemma}
\newtheorem{proposition}{Proposition}
\newtheorem{corollary}{Corollary}
\newtheorem{claim}{Claim}
\newtheorem{conjecture}{Conjecture}
\newtheorem{hypothesis}{Hypothesis}
\newtheorem{assumption}{Assumption}
\theoremstyle{EX}
\newtheorem{remark}{Remark}
\newtheorem{example}{Example}
\newtheorem{problem}{Problem}
\newtheorem{definition}{Definition}
\newtheorem{question}{Question}
\newtheorem{answer}{Answer}
\newtheorem{exercise}{Exercise}
}

\newtheorem{assump}{Assumption}
\newtheorem{remark}{Remark}
\newtheorem{lemma}{Lemma}
\newtheorem{example}{Example}
\newtheorem{theorem}{Theorem}
\newtheorem{proposition}{Proposition}

\begin{abstract}%
We revisit the classical singular control problem of minimizing running and controlling costs.
Existing studies have shown the optimality of a barrier strategy when driven by Brownian motion or \lev processes with one-sided jumps. Under the assumption that the running cost function is convex, we show the optimality of a barrier strategy for a general class of \lev processes. 
\\
\noindent \small{\noindent  {AMS 2020} Subject Classifications: Primary 60G51; Secondary 93E20, 90B05 }
\end{abstract}




\maketitle

\section{Introduction} \label{section_introduction}

In this paper,  we consider the classical stochastic control problem driven by a one-dimensional stochastic process, whose  objective is to derive an optimal strategy that minimizes the sum of running and controlling costs.  
The running cost is modeled as a function of the controlled process accumulated over time. The controlling cost is assumed to be proportional to the amount of control, meaning the considered problem falls in the class of singular control.
  As exemplified by the seminal work such as \cite{Benes}, the optimal strategy is often verified to be of barrier (or bang-bang) type. In other words, controlling the process so that it stays in an interval or a half-line is optimal. Singular control has many applications and is actively studied in financial mathematics, such as \cite{Cai,   Mundaca}. A version with ruin, called de Finetti's problem, is one focus of research in insurance mathematics \cite{AvrPalPis2007, Loeffen}. 



This paper focuses on the following formulation:
Given a state process $X = (X_t)_{t \geq 0}$, the decision-maker chooses a {strategy} $\pi = (R_t^\pi)_{t \geq 0}$ representing the aggregate amount of control until time $t$. The objective is to  minimize the total expected values of the running cost $\int_0^\infty e^{-qt} f(U_t^\pi) \diff t$ and the controlling cost  $\int_{[0, \infty)} e^{-qt} \diff R_t^\pi$ where $q >0$ is a discount factor and $U^\pi := X+R^\pi$ is the controlled process when a strategy $\pi$ is applied.  This study looks at the case when $X$ is a general \lev process and $\pi$ is increasing.


In this classical formulation, complete solutions have been obtained for several stochastic processes. A majority of research focus on the case $X$ is a Brownian motion or other diffusion processes; analytical results are obtained in more general settings, such as in \cite{Benes} (see also \cite{Guo2, Guo, Ma,Ma2, Merhi}  for other stochastic control problems for diffusion processes). The case $X$ with jumps is less studied compared to Brownian motion/diffusion models. However, extensions to spectrally one-sided \lev processes (i.e.\ \lev processes with one-sided jumps) have recently been made along with the development of the so-called fluctuation theory.
In particular,  Yamazaki \cite{Yam2017} solved the problem when $X$ is a general spectrally negative \lev process.  Several variations, such as \cite{Hernandez, Perez_Yamazaki_Bensoussan}, consider optimality over restricted sets of admissible {strategies}.
These papers commonly assume the convexity (or a slightly more relaxed condition) of the running cost function $f$ and show that a barrier strategy reflecting the state process at a certain boundary is optimal.  It is a natural conjecture that the same conclusion can be drawn for a wider class. 
%
In this paper, we verify the conjecture that a barrier strategy is optimal for a general \lev process. 
This generalization is particularly important in finance, where the assumption of the one-sided jumps  does not suit  well. Financial asset values are empirically known to contain both positive and negative jumps (e.g., \cite{CGMY}). For realistic applications in finance, the use of stochastic processes having jumps in both directions is indeed desirable.

This paper generalizes the result in the spectrally negative \lev model \cite{Yam2017}. However, 
we take a completely different approach from  \cite{Yam2017} for the proof of optimality of a barrier strategy. In  \cite{Yam2017}, the expected cost under a barrier strategy is written in terms of the scale function. The selection of the optimal barrier and the proof of verification boil down to the analysis of the scale function.  Using the known results on the smoothness and certain martingale properties of the scale function, the optimality of the barrier strategy in \cite{Yam2017} can be shown in a direct way. However, the same methodologies cannot be used for a general \lev process because the scale function is defined only for spectrally one-sided \lev processes and not for a general \lev process.

We show that the optimal barrier, which we call $b^*$ in this paper, can be characterized concisely as the solution to, with $f'_+$ the right-hand derivative of $f$,
\begin{align} \bE_b \left[\int_0^\infty e^{-qt} f^\prime_+(U^{b}_t)\diff t\right] +C = 0 \label{opt_barrier}
\end{align}
where, under $\p_b$, $U^b = (U^b_t)_{t \geq 0}$ is the reflected \lev process with lower barrier $b$ of the \lev process $X$ starting at $b$, and $C \in \mathbb{R}$ is the unit cost/reward of controlling. Thanks to the convexity assumption of $f$, the left-hand side of  \eqref{opt_barrier} is  monotone in $b$, and hence the root $b^*$ can be obtained easily by bisection. 

To show the optimality of the proposed strategy, the key observations are that the derivatives of the expected running cost $\bE_x\left[ \int_0^\infty e^{-qt} f(U^b_t)\diff t   \right]$ and controlling cost
 $\bE_x [ \int_{[0,\infty)} e^{-qt}  \diff R^b_t ]$, 
with respect to the starting value $x$ and barrier $b$, can be written concisely in terms of the first down-crossing time of $X$:
$\tau_0^- := \inf \{ t > 0: X_t < 0 \}$. 
Similar observations were used in the optimal dividend problem with capital injections in \cite{Nob2019}. However, to the best of our knowledge, this is the first result applied to the singular control problem with running costs. 
The conjectured optimal strategy is rigorously verified by showing that the candidate value function satisfies certain smoothness conditions and solves the variational inequalities.  To this end, we first show that the strategy of reflecting the process at $b^*$ is optimal among barrier strategies. Furthermore, its optimality over all admissible strategies is shown via contradiction arguments by adapting the techniques of \cite{AvrPalPis2007}.

This paper primarily aims to show the optimality for a general \lev process without focusing on a particular type. This includes cases when $X$ has paths of bounded and unbounded variations and, thus, can be applied in various settings. Thanks to the explicit and concise expression of the optimal barrier $b^*$ as a solution  to \eqref{opt_barrier}, it can be computed generally via a standard Monte Carlo simulation.

The rest of the paper is organized as follows. Section \ref{Sec02} formulates the problem. Section \ref{section_barrier} defines the barrier strategy and shows the optimality of that with barrier $b^*$ over the set of barrier strategies. Section  \ref{section_verification} shows its optimality over all admissible strategies. The discussions until Section \ref{section_verification} are based on the assumption that the state process is not a driftless compound Poisson process; this is completely relaxed in Section \ref{sec_compound_poisson_case}.  Section \ref{section_conclusion} concludes the paper.
Several long proofs are deferred to the appendix. Numerical results are provided on the author's website.\footnote{\url{https://sites.google.com/site/kyamazak/}}

\section{Preliminalies}\label{Sec02}

\subsection{Problem}

We let $(\Omega , \mathcal{F}, \bP)$ be  a probability space hosting a \lev process $X=\{X_t : t \geq 0 \}$.  
For $x\in\bR$, we denote by $\bP_x$ the law of $X$ when its initial value is $x$. In particular, we denote $\mathbb{P} = \mathbb{P}_0$.  
Throughout the paper, let $\Psi$ be the \emph{characteristic exponent} of $X$ that satisfies 
$
e^{-t\Psi(\lambda)}=
\bE \sbra{e^{i\lambda X_t}}$, for  $\lambda \in \bR$ and $t\geq 0, 
$ which is known to admit the form 
\begin{align*}
\Psi (\lambda) := -i\gamma\lambda +\frac{1}{2}\sigma^2 \lambda^2 
+\int_{\bR \backslash \{0\} } (1-e^{i\lambda z}+i\lambda z1_{\{\absol{z}<1\}}) \Pi(\diff z) , ~~~~~~\lambda\in\bR, 
\end{align*}
where $\gamma\in\bR$, $\sigma\geq 0$, and $\Pi $ is a L\'evy measure on $\bR \backslash \{0\}$ satisfying the integrability condition: 
$
\int_{\bR\backslash \{ 0\}}(1\land z^2) \Pi (\diff z) < \infty. 
$ Recall that the process $X$ has bounded variation paths if and only if $\sigma= 0$ and $\int_{\absol{z}<1} \absol{z}\Pi(\diff z)<\infty$. 


We consider the stochastic control problem with proportional controlling costs (without fixed costs).  
Let $\mathbb{F} := \{ {\mathcal{F}_t}: t \geq 0 \}$ be the natural filtration generated by $X$.
A strategy, representing the cumulative amount of controlling, $\pi=\{ R^\pi_t: t\geq 0\}$, is a nondecreasing, right-continuous, and $\mathbb{F}$-adapted process starting at $R^{\pi}_{0-}=0$.  The corresponding controlled process $U^\pi$ becomes
\begin{align}
U^\pi_t=X_t +R^\pi_t,\quad t\geq 0. \label{def_U_pi}
\end{align}
\par
We fix a discount factor $q > 0$ and a unit cost/reward of controlling $C \in \bR$ {(cost if {it is} positive and reward if negative)}. Associated with each {strategy} $\pi \in \cA$, the running cost is modeled by $\int_0^\infty e^{-qt} f(U^\pi_t)\diff t $ for a measurable running cost function $f : \bR \to\bR$
 and that of controlling is given by $C\int_{[0,\infty)} e^{-qt}\diff R^\pi_t$. The problem is to minimize their expected sum
\begin{align}
v_\pi (x):=\bE_x \left[ \int_0^\infty e^{-qt}f(U^\pi_t)\diff t + C\int_{[0, \infty)}e^{-qt} \diff R^\pi_t \right], \quad x\in\bR,  \label{v_pi}
\end{align}
over the set of all admissible {strategies} $\cA$ that satisfy all the constraints described above and {the integrability condition:}
\begin{align}
\bE_x \left[ \int_0^\infty e^{-qt}\absol{f(U^\pi_t)}\diff t + \int_{[0, \infty)}e^{-qt} \diff R^\pi_t \right]< \infty,\quad x\in\bR. \label{2}
\end{align}
The problem is to compute the value function 
\begin{align}
v(x):=\inf_{\pi\in\cA}v_\pi(x),\quad x\in\bR, \label{1}
\end{align}
and to obtain the optimal {strategy} $\pi^\ast$ that attains it, if such a {strategy} exists.


\subsection{Standing assumptions}

Throughout the paper, we assume the following on the function $f$ {and the unit cost/reward $C$}.  Note that this is commonly assumed in the literature (see, e.g., \cite{Benkherouf, Bensoussan_Liu_Sethi, Hernandez, Yam2017}).

\begin{assump}[Assumption on $f$ and $C$] \label{Ass201}
We assume the following throughout the paper:
\begin{enumerate}
\item The function $f$ is convex. This guarantees that the right- and left-hand derivative $f^\prime_+ (x)$ and $f^\prime_- (x)$, respectively,  exist for all $x \in \R$.
\item The function $f$ has at most polynomial growth in the tail.  That is to say, there exist $k_1$, $k_2$ and $N\in\bN$ such that $|f(x)|\leq k_1+k_2\absol{x}^N$ for all $x \in\bR$. 
\item  We have $f^\prime_+ (-\infty)<-Cq< f^\prime_+(\infty)$ where $f^\prime_+(-\infty):=\lim_{x\to-\infty} f^\prime_+ (x)\in[-\infty,\infty)$ and  $f^\prime_+(\infty):=\lim_{x\to\infty}f^\prime_+(x)\in(-\infty,\infty]$, {which exist by (i).} 
\end{enumerate}
\end{assump}
{
\begin{remark}
Assumption \ref{Ass201}(iii) is required so that the optimal strategy becomes non-trivial. 
\begin{enumerate}
\item Suppose $-Cq \geq f^\prime_+ (\infty)$.
If an amount of $\Delta x$ is increased, then the cost of this action is 
$C \Delta x \leq -  {f^\prime_+ (\infty)}  \Delta x / q= \Delta x \int_0^\infty e^{-qt} \inf_{ y \in \mathbb{R} } (-f^\prime_+ (y)) \diff t$,
meaning that the cost of modifying by $\Delta x$ is always smaller than the resulting reduction of running cost. Consequently, one should take $\Delta x$ arbitrarily large, and an optimal strategy does not exist ($b^* = \infty$).  
\item Suppose $f^\prime_+(-\infty) \geq  -Cq$. Then,
$C \Delta x \geq -  {f^\prime_+ (-\infty)}  \Delta x / q= \Delta x \int_0^\infty e^{-qt} \sup_{ y \in \mathbb{R} }  (-f^\prime_+ (y)) \diff t$,
meaning that the cost of modifying by $\Delta x$ is always higher than the resulting reduction of running cost.  Consequently, $\Delta x$ should be always kept zero.
\end{enumerate}
\end{remark}
}


\begin{assump}[Assumption on $X$]  \label{Ass101}
We assume the following regarding the \lev process $X$. 
\begin{enumerate}
\item 
The process $X$ is not a {driftless} compound Poisson process. 
\item There exists $\bar{\theta}>0$ such that 
$\int_{\bR\backslash(-1,1)}e^{\bar{\theta} \absol{z}}\Pi(\diff z) < \infty$. 
{This and \cite[Theorem 3.6]{Kyp2014} guarantee that $\bE [ e^{\bar{\theta} \absol{X_1}} ] <\infty$ and, since $e^{x}\geq x$ for $x\geq0$, we also have $\bE \sbra{\absol{X_1}}<\infty$.}
\end{enumerate}
\end{assump}

Assumption \ref{Ass101}(i) is assumed only temporarily until the end of Section \ref{section_verification}. However, we will show that the main results hold in the same way for the driftless compound Poisson case in Section \ref{sec_compound_poisson_case}.

For $b \in \bR$, we write the first down-crossing time: 
\begin{align} \label{def_tau_b}
\tau^-_b :=\inf\{ t>0: X_t < b \}.
\end{align}
{As addressed in the introduction, our key tools are the expressions of the derivatives of the expected costs under the barrier strategy in terms of this random variable. }

By Assumption \ref{Ass101}(i), we have the following lemma {(see Section \ref{sec0A} for its proof)}. 
\begin{lemma} \label{Rem201A}
We have the following three facts. 
\begin{enumerate}
\item$0$ is regular for $\bR \backslash \{0\}$ {(i.e.\ $\p (\inf \{ t > 0: X_t \neq 0\}=0)=1$).}
\item For fixed $x \in \bR \backslash\{0\}$, we have $\lim_{b\rightarrow x}\tau^-_b=\tau^-_x$ {$\bP$-a.s.}  {When $x = 0$,} we have $\lim_{b \uparrow 0}\tau^-_b=\tau^-_0$ {(i.e.\ $\tau^-_b$ is left-continuous at $0$) $\bP$-a.s.} 
\item The map $x \mapsto \bE_x  [e^{-q\tau^-_0} ]$ is continuous on $\bR\backslash \{0\}$ and right-continuous at $0$.
\end{enumerate}
 \end{lemma}
\begin{remark}
\par
Regarding Lemma \ref{Rem201A}, for (i), for the case $X$ is of unbounded variation, it leaves zero immediately by its rapid fluctuation, and for the case $X$ is of bounded variation, it leaves zero immediately to the direction of the drift. Note that when it is a driftless compound Poisson process, it stays at zero for a positive amount of time and hence (i) fails to hold. Lemma \ref{Rem201A} (i) and (ii) are implied by the fact that for the case $X$ is of unbounded variation or has a negative drift, the process goes below zero immediately after it hits zero; otherwise it is impossible for the process to jump onto zero. Note that these behaviors are not guaranteed for driftless compound Poisson processes.

Above discussions are for the process $X$ and not for the control process $R^\pi$. This paper focuses on the case  $R^\pi$ is monotone. However, for the case the monotonicity assumption is relaxed, the path variation of  $R^\pi$ must be carefully assumed. See, for example, \cite{BY} in the spectrally negative \lev model for the case $R^\pi$ is of bounded variation without the monotonicity assumption.
\end{remark}
\begin{remark}\label{Rem201}
By Assumption \ref{Ass101}(ii) and the polynomial growth assumption as in Assumption \ref{Ass201}(ii), we have 
$\bE_x\left[\int_0^\infty e^{-qt} \absol{f(X_t)} \diff t\right]<\infty$ for all $x\in\bR$, and the map $x\mapsto \bE_x\left[\int_0^\infty e^{-qt} \absol{f(X_t)} \diff t\right]$ is {of} polynomial growth as $x \uparrow \infty$ or $x \downarrow -\infty$. 
For its proof, see the proof of \cite[Lemma 11]{Yam2017}. 
\end{remark}

Assumption \ref{Ass101}{(ii)} is also needed in the verification step when taking limits (see Theorem \ref{Thm201}).

\section{Barrier {strategies}} \label{section_barrier}

As is commonly pursued in classical singular control problems,  our objective is to show the optimality of a barrier {strategy}
$\pi^b$ for some $b\in\bR$ with
\begin{align}
R^b_t:=R^{\pi^b}_t=-\inf_{s\in[0,t]}\cbra{(X_s-b)\land0},\quad t\geq 0. 
\end{align}
It is well known that the resulting controlled process
\[
U^b_t:=U^{\pi^b}_t = X_t + R^b_t, \quad t\geq 0,
\]
 becomes the reflected L\'evy process with a lower barrier $b$. For the rest of the paper, let
\begin{align} \label{def_v_b}
v_b (x):= v_b^{(1)}(x)  + C v_b^{(2)}(x), 
\quad x\in\bR, 
\end{align}
where we define
\begin{align}
v_b^{(1)}(x) := \bE_x\left[ \int_0^\infty e^{-qt} f(U^b_t)\diff t   \right], \quad
v_b^{(2)}(x) := \bE_x\left[ \int_{[0,\infty)} e^{-qt}  \diff R^b_t  \right],
\end{align}
{whose finiteness is verified in Lemma \ref{Thm201}.}
Note by the strong Markov property that
\begin{align}
v_b (x):= C (b-x) + v_b(b), \quad x \leq b. \label{v_b_below}
\end{align}

The proofs of the following lemmas are deferred to Appendix \ref{Sec0A1},  \ref{AppA03}, and \ref{Sec0B}. 

\begin{lemma}\label{Thm201}
For $b\in\bR$, the {strategy} $\pi^b$ satisfies {the integrability condition} \eqref{2} and is hence admissible.  
\end{lemma}

The function $v_b(x)$ grows linearly {as $x \downarrow -\infty$} {as in \eqref{v_b_below}}. In addition, we have the following.
\begin{lemma}\label{Rem202}
The function $v_b(x)$ grows at most polynomially as $x\uparrow\infty$. 
\end{lemma}

\begin{lemma}\label{Lem201}
For {$x, b \in\bR$,} we have 
$\bE_x \left[ \int_0^\infty e^{-qt}\absol{f^\prime_+(U^{b}_t)}\diff t\right]< \infty$. 
\end{lemma}

{As discussed in the introduction, the} objective of this current paper is to show that {a barrier strategy with the barrier }
\begin{align}
\begin{split}
b^\ast := \inf\left\{b\in\bR: {\rho(b)} +C \geq 0\right\}  
\end{split}
\label{def_b_star}
\end{align}
where
\begin{align} \label{def_rho}
\rho(b) := \bE_b \left[\int_0^\infty e^{-qt} f^\prime_+ (U^{b}_t)\diff t\right]=
\bE \left[\int_0^\infty e^{-qt} f^\prime_+(U^{0}_t+b)\diff t\right]
\end{align}
is optimal.
{To see that this is well-defined and finite, by} the convexity of $f$, the mapping $\rho(b)$ 
 is {nondecreasing}. 
In addition, {by Lemma \ref{Lem201},} monotone convergence gives $\lim_{b\uparrow\infty} {\rho(b)}
= { f^\prime_+(\infty)}/q $ and 
$\lim_{b\downarrow-\infty} {\rho(b)}
= f^\prime_+(-\infty) / q$. 
Hence, by Assumption \ref{Ass201}(iii), we have $-\infty <  b^\ast<\infty$. 

\begin{example} \label{example_periodic} 
For the case $f(x)  = x^2$ {({and hence } $f'(x) = 2x$)}, because
$\rho(b) = 2 ( \bE \left[\int_0^\infty e^{-qt} U^{0}_t\diff t\right] + b/q )$, we have
$b^* = {-q \big( \frac C 2 + \bE \left[\int_0^\infty e^{-qt} U^{0}_t\diff t\right] \big)}$.
\end{example}

\begin{remark} \label{remark_rho_different_exp}
By the duality of L\'evy processes, for each $t \geq 0$, the random variable $U_t^0 = X_t - \inf_{s \in [0,t]} X_s$ is the same in distribution with $\sup_{s \in [0,t]} X_s$ (as in Lemma 3.5 of \cite{Kyp2014}). Hence, we can write
\[
\rho(b) = \mathbb{E} \Big[ \int_0^\infty e^{-qt} f_+'(U_t^0 + b)\Big] = q^{-1}\mathbb{E} \Big[  f_+'(U_{\mathbf{e}_q}^0  + b)\Big] = q^{-1}\mathbb{E} \Big[  f_+'(\sup_{s \in [0,\mathbf{e}_q]} X_s + b)\Big], 
\]
where $\mathbf{e}_q$ is an independent exponential random variable with parameter $q$.
Therefore, if the Wiener-Hopf factorization is known, the Wiener-Hopf factor (i.e.\ the Laplace transform of $\sup_{s \in [0,\mathbf{e}_q]}X_s$) can be inverted to obtain the distribution of  $\sup_{s \in [0,\mathbf{e}_q]} X_s$. In particular, if $f'$ is a polynomial, more direct approach is possible by computing the moments of $\sup_{s \in [0,\mathbf{e}_q]} X_s$. There are only  few cases where the Wiener-Hopf factorization is explicitly known, but for example when the L\'evy measure is phase-type \cite{Asmussen} or meromorphic \cite{KKP}, the Wiener-Hopf factorization is analytical known and can be written as  a rational function, which can be inverted by partial fraction decomposition. 
\end{remark}

\begin{remark}  \label{remark_rho_variants}
One important fact of the characterization of $b^*$ as a root of \eqref{def_rho} is that in variations of the problem with restricted sets of admissible strategies, the optimal barrier is characterized in the same way by replacing the (classical) reflected process $U^{b}_t$ with variants of reflected processes.
 In \cite{Hernandez} where they consider the case the control process is assumed to be absolutely continuous, the optimal barrier is expressed in the same way with $U^{b}_t$ replaced by the so-called refracted process.  In \cite{Perez_Yamazaki_Bensoussan} with Poissonian control opportunities, the same holds with $U^{b}_t$ replaced by the so-called Parisian-reflected processes.
 
\end{remark}

The following lemma implies that 
\begin{enumerate}
\item[(a)] when $X$ is of unbounded variation,  $b^*$ is the unique root of $\rho(\cdot) + C = 0$;
\item[(b)] when $X$ is not the negative of a subordinator, then $\rho(b) +C < 0$ for $b < b^*$ and $\rho(b) +C> 0$ for $b > b^*$.
\end{enumerate}
\begin{lemma}\label{Lem402} 
\begin{enumerate}
\item If  $X$ has unbounded variation paths, then the function $\rho$ is continuous.
\item 
{If  $X$ is not the negative of a subordinator,} we have $\rho(b) +C > 0$ for all $b > b^*$.
\end{enumerate}
\end{lemma}
\begin{proof}
{(i) Because the function $\rho$ is right-continuous from its definition {as in \eqref{def_rho}}, } it is sufficient to prove that $\rho$ is left-continuous. 
By the dominated convergence theorem and since the set of discontinuous points of $f^\prime_+$ is at most {countably many} {(from convexity {as in Assumption \ref{Ass201}(i)})}, for $b\in\bR$ and $\varepsilon>0$, we have 
\begin{align}
\lim_{\varepsilon\downarrow0}\left(\rho(b)-\rho(b-\varepsilon)\right)
&=\lim_{\varepsilon\downarrow0}\bE \left[\int_0^\infty e^{-qt} \left(f^\prime_+(U^{0}_t+b)-f^\prime_+(U^{0}_t+b-\varepsilon)\right)\diff t\right]\\
&=\bE \left[\int_0^\infty e^{-qt} \left(f^\prime_+(U^{0}_t+b)-f^\prime_-(U^{0}_t+b)\right)\diff t\right]\\
&= \sum_{y\in [0, \infty)}\left( f^\prime_+(y +b)-f^\prime_- (y +b)\right)R^{(q)}_{U^0} (\{y \} ), \label{b013}
\end{align}
where $R^{(q)}_{U^0}$ is the potential measure given by 
$R^{(q)}_{U^0} (\diff y ):=\bE\left[\int_0^\infty e^{-qt} 1_{\{U^0_t\in\diff y\}} \diff t\right]$, $y \geq 0$.
We prove that the measure $R^{(q)}_{U^0}$ does not have a mass {and hence} \eqref{b013} is equal to $0$. 
First, we have 
$R^{(q)}_{U^0}(\{0\})
=\bE\left[\int_0^\infty e^{-qt} 1_{\{X_t= \inf_{0\leq s\leq t} X_s \} }   \diff t\right]=0$,
where in the last equality, we used \cite[Theorem 6.7]{Kyp2014} {and our assumption that $X$ is of unbounded variation}. 
Second, we prove that for $y>0$, $R^{(q)}_{U^0}(\{y\})=0$. 
{To see this, we} {let $\un{T}^{(0)}=0$ and inductively define}  
$\bar{T}^{(n)}=\inf\{t>\un{T}^{(n-1)} : U^0_t>\frac{y}{2}\}$ and $\un{T}^{(n)}=\inf\{t>\bar{T}^{(n)} : U^0_t=0\}$, $n \geq 1$.
Then, {by the strong Markov property,}
\begin{align*}
R^{(q)}_{U^0}(\{y\})
=\sum_{n\in\bN} \bE \left[ \int_{{\bar{T}^{(n)}}}^{{\un{T}^{(n)}}}  
e^{-qt}1_{\{U^0_t=y\}} \diff y \right]
=\sum_{n\in\bN} \bE \left[ e^{-q \bar{T}^{(n)}}\bE_{U^0_{\bar{T}^{(n)}}}\left[ \int_{0}^{\tau^-_0}e^{-qt}1_{\{X_t=y\}} \diff y \right]\right]=0,
\end{align*}
{where the last equality holds because} the potential of $X$ does not have a mass by \cite[Proposition I.15]{Ber1996}.
{Substituting this in \eqref{b013}, we see that $\rho$ is left-continuous, as desired.}
\par
(ii) 
By the assumption that $X$ is not the negative of a subordinator, {its reflected process} $U^0$ can exceed any level almost surely,
and thus the function $\rho$ is strictly increasing on $(-\infty , \bar{b})$ where 
\[
\bar{b}:=\inf \{ b\in\bR: f^\prime_+ \text{ is constant on }(b, \infty)\}.
\]
When $\bar{b} = \infty$ the proof is complete and hence we assume below that $\bar{b}<\infty$.

We have $f^\prime_+(x)=f^\prime_+(\infty) := \lim_{y \to \infty}f^\prime_+(y)   <\infty$ for $x\geq {\bar{b}}$. 
Thus, we have for $b\geq \bar{b}$, 
\begin{align}
\rho(b)+C=\int_0^\infty e^{-qt}f^\prime_+(\infty)\diff t+C=\frac{f^\prime_+(\infty)}{q}+C>0,\label{b014}
\end{align}
where in the last inequality, we used Assumption \ref{Ass201}(iii).
If $b^\ast\geq \bar{b}$, then the above inequality holds for $b = b^*$ and hence $\rho(b^\ast)+C
>0$ (hence $\rho(b)+ C > 0$ for all $b > b^*$ by the monotonicity of $\rho$).
 If $b^\ast < \bar{b}$, this means $\rho(\cdot)$ is strictly increasing at $b^\ast$), and hence $\rho(b)+C > 0$ for all $b > b^*$.
\end{proof}

\subsection{Optimality over barrier strategies}

We shall first show the optimality of the barrier strategy $\pi^{b^*}$ over all barrier strategies 
$\mathcal{A}_{bar} {:= \{ (R^b_t)_{t \geq 0}; b \in \R\}}$,
which is a subset of $\cA$.
\begin{theorem} \label{Thm202}
{The strategy $\pi^{b^*}$ is optimal over $\mathcal{A}_{bar}$. In other words,}
$v_{b^\ast } (x) \leq v_b(x)$ for $x, b\in\bR$.
\end{theorem}
The rest of this subsection is devoted to the proof of Theorem \ref{Thm202}.


\par

The following lemma {gives key expressions of}  the right-hand derivative {of $v_b$} with respect to the barrier $b$, {written in terms of the {first down-crossing} time \eqref{def_tau_b}}. 

\begin{lemma}\label{Lem202}


Fix $x \in \R$. 

(i) {The function} $b \mapsto v_{b}^{(1)} (x) $ is continuous on $\R$. In particular, for $b \neq x$, the right-hand derivative of $b \mapsto v_b^{(1)}(x)$ exists and
$\lim_{\varepsilon \downarrow 0} (v^{(1)}_{{b+\varepsilon} }   (x) - v^{(1)}_{b} (x) )/\varepsilon
=\bE_x \left[ \int_{\tau^-_b}^\infty e^{-qt} f^\prime_+(U^{b}_t)\diff t\right]$.

(ii) {The function}  $b \mapsto v_{b}^{(2)} (x) $ is continuous on $\R$. In particular, for $b \neq x$, the right-hand derivative of $b \mapsto v_b^{(2)}(x)$ exists and
$\lim_{\varepsilon \downarrow 0} (v^{(2)}_{{b+\varepsilon} }   (x) - v^{(2)}_{b} (x) ) / \varepsilon
= \bE_x \left[ e^{-q\tau^-_{b}} \right]$.

(iii) 
{The function} $b \mapsto v_{b} (x) $ is continuous on $\R$. In particular, for $b \neq x$,   we have 
\begin{align}
\lim_{\varepsilon \downarrow 0}\frac{v_{{b+\varepsilon} }   (x) - v_{b} (x)  }{\varepsilon}
=\bE_x \left[ \int_{\tau^-_b}^\infty e^{-qt} f^\prime_+(U^{b}_t)\diff t\right] +C\bE_x \left[ e^{-q\tau^-_{b}} \right]. \label{10}
\end{align}
\end{lemma}
\begin{proof}
{Fix 
$\varepsilon>0$.} 
Here, we focus on the pathwise behaviors of $R^{b+\varepsilon}$, $R^b$, $U^{b+\varepsilon}$ and $U^b$. 
For $t\in[0, \tau^-_{b+\varepsilon})$ {at which $X_t \geq b+ \varepsilon > b$}, we have 
\begin{align}
U^{b+\varepsilon}_t =U^b_t=X_t, \quad R^{b+\varepsilon}_t=R^b_t=0. \label{6}
\end{align}
For $t\in [\tau^-_{b+\varepsilon} , \tau^-_b)$ at which $b \leq \inf_{s \in [0, t]}X_s \leq b+\varepsilon$, 
we have
\begin{align}
\label{45}&R^{b+\varepsilon}_t=-\inf_{s\in [0, t]}(X_s-(b+\varepsilon))\leq \varepsilon , \quad R^b_t=0, \\
&U^{b+\varepsilon}_t =X_t+R^{b+\varepsilon}_t, \quad U^b_t=X_t,  
\end{align} 
implying
\begin{align}
\label{7}0\leq U^{b+\varepsilon}_t-U^b_t\leq \varepsilon. 
\end{align}
For $t \in [\tau^-_b ,\infty)$ {at which $\inf_{s \in [0, t]}X_s \leq b < b+\varepsilon$,} we have 
\begin{align}
&\label{13}R^{b+\varepsilon}_t=-\inf_{s\in [0, t]}(X_s-(b+\varepsilon))= -\inf_{s\in [0, t]}(X_s-b)+\varepsilon=R^b_t+\varepsilon, \\
&\label{8}
U^{b+\varepsilon}_t =X_t+R^{b+\varepsilon}_t
=X_t+R^{b}_t+\varepsilon =U^b_t+\varepsilon. 
\end{align}
{In addition, by replacing $b$ and $b+\varepsilon$ with $b-\varepsilon$ and $b$, respectively, we obtain, for $t \geq 0$, $0\leq U^b_t-U^{b-\varepsilon}_t \leq \varepsilon$ and $R^b_t-R^{b-\varepsilon}_t \leq \varepsilon$. Putting these together, for any $\varepsilon \in \R$, we have {for all $t \geq 0$}
\begin{align}
&0\leq |U^{b+\varepsilon}_t-U^{b}_t| \leq {|\varepsilon|} \quad \textrm{and} \quad 0\leq |R^{b+\varepsilon}_t-R^{b}_t| \leq {|\varepsilon|}. \label{49}
\end{align}}
{(i)} 
{From 
\eqref{49},} the convexity of $f$ and Lemma \ref{Lem201}, 
for $b\in\bR$ and $\varepsilon\in(-1, 1)$ 
the triangle inequality and the mean value theorem {give}
\begin{align}
\left| v_{b+\varepsilon}^{(1)}(x) - v_{b}^{(1)}(x)\right|
\leq&\bE_x \left[\int_0^\infty e^{-qt} \left|f(U^{b+\varepsilon}_t)-f(U^{b}_t)  \right| \diff t\right] \\ \leq& \bE_x \left[\int_0^\infty e^{-qt} \left|U_t^{b+\varepsilon} -U^{b}_t  \right| \sup_{U_t^b - 1 < y < U_t^b + 1} |f_+'(y) | \diff t\right] \\
\leq&\left| \varepsilon\right|\bE_x \left[\int_0^\infty e^{-qt} \left( \left|f^\prime_+(U^b_t-1) \right| \lor \left|f^\prime_+(U^b_t+1) \right| \right) \diff t\right] {\xrightarrow{\varepsilon \downarrow 0} 0,}
\end{align}
{implying} that $b\mapsto  v_{b}^{(1)}(x)$ is continuous for all $x \in\bR$.

{Fix $\varepsilon > 0$ and $b \neq x$.}
{Let}
\begin{align*}
{A_b^{(1)}(\varepsilon) :=} \frac{v_{b+\varepsilon}^{(1)}(x)-v_{b}^{(1)}(x)}{\varepsilon}=
\bE_x \left[ \int_0^\infty e^{-qt}\frac{f(U^{b+\varepsilon}_t)-f(U^{b}_t) }{\varepsilon}\diff t\right].
\end{align*}
By \eqref{6} and \eqref{8}, we have 
\begin{align*}
\int_{0}^\infty e^{-qt}\frac{f(U^{b+\varepsilon}_t)-f(U^{b}_t) }{\varepsilon}\diff t =
\int_{\tau^-_{b+\varepsilon}}^{\tau^-_{b}} e^{-qt}\frac{f(U^{b+\varepsilon}_t)-f(U^{b}_t) }{\varepsilon}\diff t
+\int_{\tau^-_{b}}^\infty e^{-qt}\frac{f(U^{b}_t+\varepsilon)-f(U^{b}_t) }{\varepsilon}\diff t. 
\end{align*}
Here, since $f$ is convex and by \eqref{7}, 
we have  
\begin{align}
\int_{\tau_{b+\varepsilon}^{-}}^{\tau_b^{-}} e^{-qt} \frac {|f(U^{b+\varepsilon}_t)-f(U^{b}_t)|} \varepsilon \diff t  \leq \int_{\tau_{b+\varepsilon}^{-}}^{\tau_b^{-}} e^{-qt} 
\big( |f^\prime_+(U^{b}_t)|\lor|f^\prime_+(U^{b+\varepsilon}_t)| \big) \diff t \xrightarrow{\varepsilon \downarrow 0} 0,
\end{align}
{where the last limit holds because $\tau_{b+\varepsilon}^- \xrightarrow{\varepsilon \downarrow 0} \tau_b^-$ {$\p_x$-}a.s.\ {for $x \neq b$} from Lemma \ref{Rem201A}(ii).}
{This limit together with} Lemma \ref{Lem201} and the dominated convergence theorem {shows},  for $x, b\in\bR$ with $x\neq b$, 
$ \lim_{\varepsilon\downarrow0}
{A_b^{(1)}(\varepsilon)} 
=
\bE_x \left[ \int_{\tau^-_b}^\infty e^{-qt} f^\prime_+(U^{b}_t)\diff t\right],
$ {as desired.}

\par 
(ii) If $\varepsilon > 0$, from \eqref{6}, \eqref{45} and \eqref{13}, $t \mapsto R^{b+\varepsilon}_t-R^{b}_t$ stays at $0$ on {$[0, \tau_{b+\varepsilon}^-)$},
increases to $\varepsilon$ on {$[\tau_{b+\varepsilon}^-, \tau_{b}^-)$} and stays at $\varepsilon$ afterwards, and therefore
\begin{align}
\varepsilon\bE_x \left[ e^{-q\tau^-_{b}} \right]\leq\bE_x \left[ \int_0^\infty e^{-qt}  \diff (R^{b+\varepsilon}_t-R^{b}_t) \right] \leq \varepsilon\bE_x \left[ e^{-q\tau^-_{b+\varepsilon}} \right]. \label{bound_R_difference}
\end{align}
{Analogous bound holds for the case $\varepsilon < 0$. Hence,}
the map $b\mapsto  v_{b}^{(2)}(x)$ is continuous for all $x \in\bR$ since 
$| v_{b+\varepsilon}^{(2)}(x)-v_{b}^{(2)}(x) |
{=}  \left|\bE_x \left[ \int_0^\infty e^{-qt}  \diff (R^{b+\varepsilon}_t-R^{b}_t) \right] \right| \leq|\varepsilon| {\xrightarrow{|\varepsilon| \downarrow 0} 0}$. 

Now let $\varepsilon > 0$ and define
\begin{align}
A_b^{(2)}(\varepsilon) :=  \frac{v_{b+\varepsilon}^{(2)}(x)-v_{b}^{(2)}(x)}{\varepsilon}=\frac{1}{\varepsilon}
\bE_x \left[ \int_0^\infty e^{-qt}  \diff (R^{b+\varepsilon}_t-R^{b}_t) \right].  \label{9}
\end{align}
For $x, b\in\bR$ 
with $x\neq b$, by {Lemma} \ref{Rem201A}(iii) and \eqref{bound_R_difference},
$\lim_{\varepsilon\downarrow0}{A_b^{(2)}(\varepsilon)} 
= \bE_x \left[ e^{-q\tau^-_{b}} \right], 
$ {as desired.}
\par
Finally, (iii) is immediate by (i) and (ii).
\end{proof}
{We are now ready to complete the proof of Theorem \ref{Thm202}.}
\begin{proof}[Proof of Theorem \ref{Thm202}.]
By {\eqref{10} and} the strong Markov property {and the fact that $U_{\tau_b^-} = b$ on $\{ \tau_b^- < \infty\}$},
\begin{align}
{ \lim_{\varepsilon \downarrow 0}\frac{v_{{b+\varepsilon} }   (x) - v_{b} (x)  }{\varepsilon} } =\bE_x \left[ e^{-q\tau^-_{b}} \right] \left(
{\rho(b)}
+C\right), 
\quad b \in \bR \backslash\{x\}, \label{53}
\end{align}
which is, by \eqref{def_b_star}, 
no more than zero for $b<b^\ast$ and no less than zero for $b\geq b^\ast$.
\par
{(i)} Fix $x \in \mathbb{R}$. We now prove that $b\mapsto v_b (x)$ is nonincreasing on $(-\infty , b^\ast)$ by contradiction. 

Suppose that $b\mapsto v_b (x)$ fails to be nonincreasing on $(-\infty , b^\ast)$. 
Then there exist {$a_1< a_2 < b^\ast$} such that $v_{a_1} (x)<v_{a_2} (x)$ {and} $x \not\in[a_1, a_2]$.
We define the function on $[a_1, a_2]$
\begin{align}
f^{(x)}(y):=v_{y} (x)-(y-a_1)\frac{v_{a_2} (x)-v_{a_1} (x)}{a_2-a_1}, \quad {a_1 \leq y \leq a_2.} \label{v_x_y}
\end{align} 
Since ${y \mapsto f^{(x)}(y)}$ is a continuous function on {[$a_1,a_2$]} by Lemma \ref{Lem202}{(iii)}  and since $f^{(x)}(a_1)=f^{(x)}(a_2) =v_{a_1}(x)$, there exists {a minimizer} $a_3\in [a_1, a_2)$ {such that}
\begin{align}
f^{(x)}(a_3)= \min_{y \in [a_1, a_2]} f^{(x)} (y). \label{43}
\end{align}
From Lemma {\ref{Lem202}(iii) (showing that the right-hand derivative of $f^{(x)}$ in $y$, {say $f^{(x),\prime}_+$,} exists)} and \eqref{43}, 
we have $f^{(x),\prime}_+(a_3)\geq 0$, and {consequently by \eqref{v_x_y},} $\lim_{\varepsilon \downarrow 0} (v_{{a_3+\varepsilon} }   (x) - v_{a_3} (x) ) / \varepsilon \geq (v_{a_2} (x)-v_{a_1} (x)) / (a_2-a_1) >0$. This contradicts the fact that \eqref{53} {when setting} $b=a_3 {< b^*}$ is less than $0$. 

{(ii)}  From the same argument we can prove that $b\mapsto v_b (x)$ is nondecreasing on $(b^\ast, \infty)$ for $x \in \bR$. 
\end{proof}
\subsection{Slopes and convexity of $v_b$}

\begin{lemma}\label{Lem203}
We fix $b\in\bR$. 

{(i)} The function {$x  \mapsto v_b(x)$} is continuous {on $\R$} and
 has the {right- and left-hand} derivatives given by 
\begin{align}
\label{34}
v_{b,+}^\prime(x)&=\bE_x \Big[ \int_0^{\tau^{-}_b} e^{-qt} f^\prime_+(X_t)\diff t \Big]-C\bE_x\left[e^{-q\tau^-_b}\right], \quad x\in\bR, \\
v_{b,-}^\prime(x)&=\bE_x \Big[ \int_0^{\tau^{-}_b} e^{-qt} f^\prime_-(X_t)\diff t \Big]-C\bE_{{x-}}\left[e^{-q\tau^-_b}\right], \quad x\in\bR {\backslash\{ {b}\}}, 
\end{align}
{where the left-hand limit $\bE_{{x-}} [e^{-q\tau^-_b} ] :=\lim_{y \uparrow x}\bE_{y} [e^{-q\tau^-_b} ]  $ exists because $\tau^-_b$ is monotone in the starting value of the process.}

{(ii)} In particular, for $x \neq b$, it is differentiable with $v_{b}^\prime(x) = v_{b,+}^\prime(x) = v_{b,-}^\prime(x)$.

\end{lemma}

\begin{proof}
{(i)} For $y\in \bR$ and $t \geq 0$, we write $X^{(y)}_t:=X_t+y$ {and} $\tau^{(y)}_a := \inf\{ t> 0: X^{(y)}_t < a\}$ for $a \in \bR$.
{Analogously, let $R^{(y),b}_t := 
{-\inf_{s\in[0,t]} \{(X_s^{(y)}-b)\land0} \}
$ and $U^{(y),b}_t := X^{(y)}_t + R^{(y),b}_t$ be the corresponding barrier strategy with barrier $b$ and controlled  process corresponding to $X^{(y)}$.} 
Then we have, for $\varepsilon >0$, 
\begin{align}
&\label{19}\frac{v_b^{(1)}(x+\varepsilon)-v_b^{(1)}(x)}{\varepsilon}= {\bE} \left[ \int_0^\infty e^{-qt} \frac{f(U^{(x+\varepsilon),b}_t)-f(U^{(x),b}_t)}{\varepsilon}\diff t \right], \\
&\label{20}\frac{v_b^{(2)}(x+\varepsilon)-v_b^{(2)}(x)}{\varepsilon}=\frac{1}{\varepsilon} {\bE} \left[ \int_{[0,\infty)}e^{-qt}\diff (R^{(x+\varepsilon),b}_t - R^{(x),b}_t)\right]. 
\end{align}
Here, similar to the arguments in the proof of Lemma \ref{Lem202}, we focus on the {pathwise} behaviors of $R^{(x+\varepsilon),b}$, $R^{(x),b}$, $U^{(x+\varepsilon),b}$, and $U^{(x),b}$. 
For $t\in[0, \tau^{(x)}_{b})$ {(so that $X_t^{(x+\varepsilon)} \geq  X_t^{(x)} \geq b$)}, 
we have 
\begin{align}
\label{14}U^{(x+\varepsilon),b}_t=X^{(x+\varepsilon)}_t=X^{(x)}_t+\varepsilon=U^{(x),b}_t+\varepsilon, \quad R^{(x+\varepsilon),b}_t=R^{(x),b}_t=0. 
\end{align}
For $t\in [\tau^{(x)}_{b} , \tau^{(x+\varepsilon)}_{b})$ {(so that $0 \leq \inf_{s\in[0,t]}(X_s^{(x+\varepsilon)}-b) \leq \varepsilon$),}
we have
\begin{align}
\begin{split}
&R^{(x),b}_t
{= -\inf_{s\in [0, t]}(X^{(x)}_s-b)}
{ = -\inf_{s\in[0,t]}((X_s^{(x+\varepsilon)}-b)-\varepsilon)}
\leq \varepsilon , \; R^{(x+\varepsilon),b}_t=0, \; \\
&U^{(x),b}_t =X^{(x)}_t+R^{(x),b}_t, \; U^{(x+\varepsilon),b}_t=X^{(x+\varepsilon)}_t,  
\end{split}
 \label{15}
\end{align} 
that imply 
\begin{align}
\label{16}0\leq U^{(x+\varepsilon),b}_t-U^{(x),b}_t\leq \varepsilon. 
\end{align}
For $t \in [\tau^{(x+\varepsilon)}_{b} ,\infty)$ {(so that $\inf_{s\in[0,t]}(X_s^{(x)}-b) \leq \inf_{s\in[0,t]}(X_s^{(x+\varepsilon)}-b) \leq 0$)},  we have 
\begin{align}
\label{17}&R^{(x),b}_t=
-\inf_{s\in[0,t]} (X_s^{(x)}-b)  = -\inf_{s\in [0, t]}(X^{(x+\varepsilon)}_s-b)+\varepsilon=R^{(x+\varepsilon),b}_t+\varepsilon, \\
\label{18}&
U^{(x),b}_t =X^{(x)}_t+R^{(x),b}_t
= {(X^{(x+\varepsilon)}_t - \varepsilon) +(R^{(x+\varepsilon),b}_t + \varepsilon)}=U^{(x+\varepsilon),b}_t. 
\end{align}
By \eqref{14} and \eqref{18}, 
we have 
\begin{align}
& \int_0^\infty e^{-qt} (f(U^{(x+\varepsilon),b}_t)-f(U^{(x),b}_t)) \diff t \\
 &= \int_0^{\tau_b^{(x)}} e^{-qt} (f(U^{(x),b}_t+\varepsilon)-f(U^{(x),b}_t)) \diff t  + \int_{\tau_b^{(x)}}^{\tau_b^{(x+\varepsilon)}} e^{-qt} (f(U^{(x+\varepsilon),b}_t)-f(U^{(x),b}_t)) \diff t. \label{001aa}
\end{align}
{Here, since} $f$ is convex and by \eqref{16}, 
we have 
\begin{align}
\int_{\tau_b^{(x)}}^{\tau_b^{(x+\varepsilon)}} e^{-qt} \frac {|f(U^{(x+\varepsilon),b}_t)-f(U^{(x),b}_t)|} \varepsilon \diff t  \leq \int_{\tau_b^{(x)}}^{\tau_b^{(x+\varepsilon)}} e^{-qt} 
\big( |f^\prime_+(U^{(x),b}_t)|\lor| 
{f^\prime_+(U^{(x),b}_t + \varepsilon)}| \big) \diff t \xrightarrow{\varepsilon \downarrow 0} 0,
\end{align}
{where the last limit holds because $\tau_b^{(x + \varepsilon)} \xrightarrow{\varepsilon \downarrow 0} \tau_b^{(x)}$ a.s. from Lemma \ref{Rem201A}(ii) (noting that $\tau_b^{(x + \varepsilon)} = \tau_{b-x-\varepsilon}^-$ {and $\tau_b^{(x)} = \tau_{b-x}^-$} $\p$-a.s.).}

Therefore, by 
Lemma \ref{Lem201} and the dominated convergence theorem, {taking a limit in \eqref{19} gives}
\begin{multline}
\lim_{\varepsilon\downarrow0} { \frac{v_b^{(1)}(x+\varepsilon)-v_b^{(1)}(x)}{\varepsilon} }
= \bE \left[ \int_0^{\tau^{(x)}_b} e^{-qt} f^\prime_+(U^{(x),b}_t)\diff t \right] \\
= 
\bE_x \left[ \int_0^{\tau^{-}_b} e^{-qt} f^\prime_+(U^{b}_t)\diff t \right]
=\bE_x \left[ \int_0^{\tau^{-}_b} e^{-qt} f^\prime_+(X_t)\diff t \right], \quad x, b \in \bR. \label{21}
\end{multline}

{On the other hand, } by \eqref{14}, \eqref{15}, and \eqref{17}, {the difference \eqref{20} is bounded with}
$- \bE [e^{-q\tau^{(x)}_b} ]\leq
(v_b^{(2)}(x+\varepsilon)-v_b^{(2)}(x))/\varepsilon
\leq- \bE [e^{-q\tau^{(x+\varepsilon)}_b} ]$.
By 
the dominated convergence theorem and {using again} $\tau_b^{(x + \varepsilon)} \xrightarrow{\varepsilon \downarrow 0} \tau_b^{(x)}$ a.s., we have, for $x, b\in\bR$, 
$\lim_{\varepsilon\downarrow0} (v_b^{(2)}(x+\varepsilon)-v_b^{(2)}(x)) / \varepsilon
=-\bE_x [e^{-q\tau^-_b} ]$. 
This shows for the right-hand derivative. Similar arguments show for the left-hand derivative {for $x \neq b$}.

{The above arguments also show the continuity of $x \mapsto v_b(x)$ for $x \neq b$ and hence} it remains to show the continuity  for $x = b$. In view of
\eqref{14}, \eqref{15}, \eqref{16}, \eqref{17} and \eqref{18}, by replacing $b$ and $b+\varepsilon$ with $b-\varepsilon$ and $b$, respectively, we have, for $\varepsilon\in\bR$ and $t\geq 0$, 
$0\leq |U^{(x+\varepsilon),b}_t-U^{(x),b}_t| \leq |\varepsilon|$ and $0\leq |R^{(x+\varepsilon),b}_t-R^{(x),b}_t| \leq |\varepsilon|$. 
Thus we get the continuity in the same way as {the proof of } Lemma \ref{Lem202}.


{(ii) To show the differentiability {(i.e.\ the right- and left-hand derivatives obtained in (i) coincide)}  at $x \neq b$, {with the measure
$R^{(q)}(x, \cdot):=\bE_x \left[\int_0^\infty e^{-qt}1_{\{X_t\in\cdot\}}\diff t\right]$, since the points $y\in\bR$ such that $f^\prime_{+}(y)\neq f^\prime_- (y)$ are {at most} countable, }
\begin{align}
\absol{  \bE_x \left[ \int_0^{\tau^{-}_b} e^{-qt} f^\prime_+(X_t)\diff t \right]-
\bE_x \left[ \int_0^{\tau^{-}_b} e^{-qt} f^\prime_-(X_t)\diff t \right]  }
\leq \sum_{y\in[b,\infty)}\left|  f^\prime_{+}(y)  -f^\prime_-( y)\right| R^{(q)}(x, \{y\}), 
\end{align}
%
%
which is zero by \cite[Proposition I.15]{Ber1996} {(note that this holds even for $x=b$)}. In addition, $\bE_x [e^{-q\tau^-_b}]=\bE_{x-} [e^{-q\tau^-_b}]$ for $x \neq b$ by {Lemma} \ref{Rem201A}(iii). This shows  $v_{b,+}^\prime(x) = v_{b,-}^\prime(x)$ {for $x \neq b$}  as desired.
} 
\end{proof}

{The following lemma states that with the selection $b^*$ as in \eqref{def_b_star}, the function $v_{b^*}$ is continuously differentiable {on $\R$}.}
\begin{lemma}\label{Lem204}
For $x\in\bR$, the function $v_{b^\ast}$ is convex and {belongs to} $C^1(\mathbb{R})$ with its derivative given by 
\begin{align}
v_{b^\ast}^\prime(x)=\bE_x \left[ \int_0^{\infty} e^{-qt} f^\prime_{\varepsilon^\ast}(U^{b^\ast}_t)\diff t \right] 
\label{40}
\end{align}
{where $f^\prime_{\varepsilon^*}(y):=(1-\varepsilon^*) f^\prime_+(y) +{\varepsilon^*} f^\prime_-(y)$} for $y\in\bR$, 
for some $\varepsilon^\ast \in [0,1]$ (which is invariant to $x$).
In particular, we have 
$v_{b^\ast}^\prime(x)=-C$ for $x \in (-\infty , b^\ast]$.
\end{lemma}
\begin{proof}
See Appendix \ref{Sec0D}.
\end{proof}

\section{Verification {of optimality}} \label{section_verification}


In Theorem \ref{Thm202}, we showed the optimality of $\pi^{b^\ast}$ over the set of barrier strategies {$\mathcal{A}_{bar}$}. In this section, we {strengthen the result and} show the optimality over all admissible strategies {$\mathcal{A}$}.  

{
Our main result of this paper is as follows.
\begin{theorem}\label{theorem_main_2}
Under the setting decided in Section \ref{Sec02}, {the barrier strategy} $\pi^{b^\ast}$ is optimal {over $\mathcal{A}$}.
\end{theorem}
To show this main theorem, we first assume certain smoothness of $f$ for the unbounded variation case in Section \ref{subsection_optimality_under_smoothness} and then in Section \ref{subsection_optimality_general} we completely relax this condition.
}

\subsection{First optimality result} \label{subsection_optimality_under_smoothness}



For $n = 1,2$, let $C^{n}_{\text{poly}}$ be a subset of {$n$-times continuously differentiable functions} {$g$ in} $C^n(\bR)$ satisfying, for some $b_1, b_2>0$ and $M \in\bN$, 
$
\absol{g (x)} < b_1 {\absol{x}}^{{M}}+b_2$,  $x\in\bR.  \label{211b}
$

From Lemmas \ref{Rem202} and \ref{Lem204}, {we already know} $v_{b^*} \in C^{1}_{\text{poly}}$. 
{For now, we assume the following to ensure the  smoothness for the case $X$ has paths of unbounded variation.}
\begin{assump} \label{assump_C_line} {For the case $X$ is of unbounded variation, we assume
 the running cost function $f \in C^2(\bR)$ and $f^{\prime\prime}$ has polynomial growth in the tail. }
\end{assump}

The proof {of the following} is given in Appendix \ref{subsection_proof_Lem205a}.

\begin{lemma}\label{Lem205a}
{Under Assumption \ref{assump_C_line},}
function $v_{b^\ast}$ belongs to {${C^{2}}$ (and hence $C^{2}_{\text{poly}}$).}  
\end{lemma}



%


{The rest of this subsection is devoted to the proof of Theorem \ref{theorem_main_2} under Assumption \ref{assump_C_line}. This will be relaxed in the next subsection.}





Let $\cL$ be 
the infinitesimal generator associated with the process $X$ applied to $g \in C^{1}_{\text{poly}}$ (resp., $C^{2}_{\text{poly}} $) for the case in which $X$ is of bounded (resp., unbounded) variation with 
\begin{align}
\cL  {g} (x) := \gamma g^\prime (x)+\frac{1}{2}\sigma^2g^{\prime \prime}(x)+
\int_{\bR\backslash \{0\}}(g(x+z)-g(x)-g^\prime(x)z1_{\{\absol{z}<1\}})\Pi(\diff z),~x\in \bR . \label{212b}
\end{align}
Also, we write $(\cL-q)  {g} (x) := \cL  {g} (x) - q g(x)$, and  for any c\`adl\`ag process $Y$, let $\Delta Y_t := Y_t - Y_{t-}$, $t \geq 0$.

\begin{proposition}[verification lemma]
\label{Prop201}
Suppose that $X$ has bounded (resp., unbounded) variation paths, and let  $w$ be a function 
on $\bR$ belonging to $C^{1}_{\text{poly}}$ (resp., $C^{2}_{\text{poly}}$) and satisfying, for $x \in \bR$, 
\begin{align}
(\cL -q) w (x)+f(x) \geq 0,      \quad  &\label{201}\\
w^\prime (x)+C \geq 0.\quad& \label{202}
\end{align}
Then we have $w(x)\leq v (x)$ for all $x\in\bR$. 
\end{proposition}
\begin{proof}
Let $\pi \in \cA$ be any admissible strategy {that satisfies \eqref{2}}. 
By an application of 
the It\^o formula (see, e.g., \cite[Theorem II.31 or II.32]{Pro2005}), 
with $R^{\pi, c}_t$ the continuous part of $R^\pi_t$ such that
$U^{\pi}_t = X_t +R^{\pi, c}_t + \sum_{0\leq s\leq t}\Delta R^{\pi}_s,$ $t\geq 0$, 
we have 
\begin{align*}
e^{-qt}w(U^\pi_t)-w(U^\pi_{0-})
&=\int_0^t e^{-qs}(\cL - q)w(U^\pi_{s-})\diff s
+M_t\\
&+\int_0^t e^{-qs}w^\prime (U^\pi_{s-}) \diff R^{\pi,c}_s  
{+\sum_{0\leq s\leq t} e^{-qs}\rbra{w(U^\pi_{s-} + \Delta X_s + \Delta R_s^\pi) - w(U^\pi_{s-} + \Delta X_s)}}, 
\end{align*}
where $\{M_t : t\geq 0\}$ is a local martingale satisfying 
\begin{align*}
M_t&:=\sigma \int_0^t e^{-qs} w^\prime (U^\pi_{s-}) \diff B_s
+\int_{[0,t ] \times (\bR {\backslash \{0\})}}e^{-qs}
\rbra{ w(U^\pi_{s-}+y) - w(U^\pi_{s-}) 
}(\cN (\diff s \times \diff y ) -\diff s \times\Pi(\diff y)),
\end{align*}
where $B$ is a standard Brownian motion and $\cN$ is a Poisson random measure {associated with the jumps of $X$} in the measure space 
$([0, \infty) \times\bR ,\cB [0, \infty ) \times \cB (\bR {\backslash \{0\}}), \diff s \times \Pi( \diff x)) $. 
By \eqref{202}, we have 
\begin{align}
{\int_0^t e^{-qs}w^\prime (U^\pi_{s-}) \diff R^{\pi,c}_s}
&\geq - C\int_0^t e^{-qs} \diff R^{\pi,c}_s,  \\ 
{\sum_{0\leq s\leq t} e^{-qs}\rbra{w(U^\pi_{s-} + \Delta X_s + \Delta R_s^\pi) - w(U^\pi_{s-} + \Delta X_s)}}
&\geq -C \sum_{0\leq s\leq t} e^{-qs} \Delta R^\pi_s,
\end{align}
and {together with} \eqref{201}, we have 
\begin{align}
e^{-qt}w(U^\pi_t)-w(U^\pi_{0-})
&\geq -\int_0^t e^{-qs}f(U^\pi_{s-})\diff s
+M_t-C\int_{[0, t]} e^{-qs} \diff R^{\pi}_s. 
\end{align}
Because $M$ is a local martingale, we can take a localizing sequence of stopping times ${\{T_n\}}_{n\in\bN}$  for $M$ with $T_n \uparrow \infty$ almost surely. 
Then, taking expectations, we have 
\begin{align}
w(x) &\leq\bE_x\sbra{ \int_0^{t \land T_n} e^{-qs}f(U^\pi_{s-})\diff s
+C\int_{[0, {t \land T_n}]} e^{-qs} \diff R^{\pi}_s} +\bE_x \sbra{e^{-q{(t\land T_n)}} w(U^{\pi}_{t\land T_n})}.
\end{align}
From the proof of \cite[Theorem 2]{Yam2017}, we have $\lim_{t\uparrow\infty,n\uparrow\infty}\bE_x \sbra{e^{-q{(t\land T_n)}} w(U^{\pi}_{t\land T_n})}=0$. 
By taking the  limit as $t\uparrow \infty$ and $n \uparrow \infty$ {and the dominated convergence theorem thanks to \eqref{2}}, the proof is complete. 
\end{proof}
{We now} prove that the candidate value function $v_{b^*}$ 
satisfies the conditions {\eqref{201} and \eqref{202}} in Proposition \ref{Prop201} {(the opposite inequality {$v(x) \leq v_{b^*}(x)$} holds because $\pi^{b^*}$ is admissible as in Lemma \ref{Thm201}).} 
The latter condition is immediate as follows. 

{
\begin{remark}By Lemma \ref{Lem204},
$v_{b^*}$ satisfies \eqref{202}. 
\end{remark}
}
\par
{In view of Proposition \ref{Prop201}, we are now left to show that $v_{b^\ast}$ satisfies \eqref{201}.}


\begin{lemma}\label{Lem205}
For $x \in [b^\ast, \infty)$, we have 
\begin{align}
(\cL -q) v_{b^\ast}(x)+f(x)=0. \label{25}
\end{align}
\end{lemma}
\begin{proof}
For $x \in \bR$, we write 
$\varphi_{b^\ast}(x) := \bE_x [e^{-q\tau^-_{b^\ast}}]$. 
For $x\in\bR$, because $U_{\tau^-_{b^\ast}}^{b^\ast} = b^\ast$ {on $\{ \tau^-_{b^\ast} < \infty \}$} and by the strong Markov property, 
\begin{align}
{v_{b^\ast}^{(2)}(x)} = \bE_x\left[e^{-q\tau^-_{b^\ast}}(b^\ast-X_{\tau^-_{b^\ast}}) {1_{\{\tau_{b^*}^- < \infty \}}}  \right] + \varphi_{b^\ast}(x)
{v_{b^\ast}^{(2)}(b^\ast)} . \label{v_b_1_mart}
\end{align}
For $x \in (b^\ast, \infty)$, the process $\{M^{[1]}_t  :t\geq 0 \}$ where 
$M^{[1]}_t :=e^{-q(\tau^-_{b^\ast} \land t ) } 
\varphi_{b^\ast}(X_{\tau^-_{b^\ast} \land t })$, $t\geq 0$,
is a martingale under $\bP_x$ because 
\begin{align}
\bE_x \sbra{ e^{-q\tau^-_{b^\ast} } |\cF_t}
&=\bE_x \sbra{e^{-q\tau^-_{b^\ast} } 1_{\{\tau^-_{b^\ast} \leq t \}}+e^{-q\tau^-_{b^\ast} } 1_{\{t< \tau^-_{b^\ast}\}} |\cF_t}\\
&=e^{-q\tau^-_{b^\ast} } 1_{\{\tau^-_{b^\ast} \leq t \}}+e^{-qt}  1_{\{ t<\tau^-_{b^\ast}\}}
{\varphi_{b^\ast}}\rbra{X_t}
=
{e^{-q(\tau^-_{b^\ast} \land t ) }} 
{\varphi_{b^\ast}} (X_{\tau^-_{b^\ast} \land t }).  
\end{align}
By the same argument, 
$\{M^{[2]}_t :t\geq 0 \}$ and $\{M^{[3]}_t :t\geq 0 \}$ where, for $t \geq 0$,
\begin{align}
M^{[2]}_t &:= e^{-q(\tau^-_{b^\ast} \land t ) } \bE_{{X_{\tau^-_{b^\ast}\land t }}}\left[e^{-q\tau^-_{b^\ast}}(b^\ast-X_{\tau^-_{b^\ast}}) 1_{\{\tau_{b^*}^- < \infty \}}  \right]
,\quad & \\
M^{[3]}_t &:= \int_0^{\tau^-_{b^\ast}\land t  }e^{-qs} f(X_s ) \diff s +e^{-q(\tau^-_{b^\ast} \land t)}
{v_{b^\ast}^{(1)}(X_{\tau^-_{b^\ast}\land t })},
\end{align}
are martingales under $\bP_x$  for $x\in {( b^\ast, \infty)}$.  
{By these and \eqref{v_b_1_mart},} the process $\{ M^{[4]}_t:t\geq 0 \}$ where 
\begin{align}
M^{[4]}_t :=e^{-q(\tau^-_{b^\ast}\land t ) } v_{b^\ast}(X_{\tau^-_{b^\ast}\land t})
+\int_0^{\tau^-_{b^\ast}\land t  }e^{-qs} f(X_s ) \diff s 
, \quad t\geq 0, \label{568}
\end{align}
is a martingale under $\bP_x$.
\par
Now by the same reasoning as that of the proof of \cite[(12)]{BifKyp2010},  
we have \eqref{25} for $x\in{(b^\ast , \infty)} $.  
{This also holds for $x=b^\ast$ by the continuity of  \eqref{25}, thanks to the 
smoothness of $v_{b^*}$.} 
\end{proof}

\begin{lemma}\label{Lem403}
For $x< b^\ast$, we have 
$(\cL -q)  v_{b^\ast} (x) +f(x)\geq0$. 
\end{lemma}
\begin{proof}
We follow the proof of  \cite[Lemma 5]{AvrPalPis2007} {and modify it for our results}. \par
We write $g(x) := (\cL -q)  v_{b^\ast} (x) +f(x)$. {In particular,} for $x < b^\ast$, because $ v_{b^\ast}'(x) =  v_{b^\ast}'(b^\ast) = -C$ and $v_{b^\ast}''(x) = 0$ {(so that $v_{b^*}(x) = c_{b^*} - Cx$ with $c_{b^\ast} := Cb^\ast+v_{b^\ast}(b^\ast)$)}, we can write
\begin{align}
g(x)=  \gamma v_{b^\ast}^\prime (b^\ast)+
\int_{\bR\backslash \{0\}}(v_{b^\ast}(x+z)+Cx-c_{b^\ast}-v_{b^\ast}^\prime(b^\ast)z1_{\{\absol{z}<1\}})\Pi(\diff z)
+q(Cx-c_{b^\ast})+f(x).
\end{align}
From the convexity of $v_{b^\ast}$ {(as  in Lemma \ref{Lem204})} and $f$, {the function} $g$ is convex on $(-\infty, b^\ast)$.
\par
{Next we} fix $b<b^\ast$ and show the next identity:
\begin{align}
v_b(x)-v_{b^\ast}(x)=\bE_x\left[\int_0^\infty e^{-qs}g(U^b_{s}) {1_{\{ -\infty < U^b_{s} < b^\ast  \}}} \diff s\right], \quad x\in\bR. \label{27}
\end{align}
Because $U^b$ is a semimartingale and $v_{b^*}$ is sufficiently smooth,  
we can follow similar steps as the proof of Proposition \ref{Prop201} via the It\^o formula.
With the continuous part of $R^{b}_t$ {as $R^{b,c}_t$ so that
$R^{b}_t = R^{b,c}_t +\sum_{s\in[0,t]}\Delta R^b_s$, $t\geq 0$ and {a localizing sequence } ${\{T^b_n := \inf\{t > 0: U^b_t-b>n\}\}}_{n\in\bN}$,
}
\begin{align*}
\begin{aligned}
\bE_x \left[ e^{-q{(t\land T^b_n)}}v_{b^\ast}(U^b_{t\land T^b_n})\right]-v_{b^\ast}(x)
&=\bE_x \left[ \int_0^{t\land T^b_n} e^{-qs}(\cL - q)v_{b^\ast}(U^b_{s-})\diff s \right]
+\bE_x \left[ \int_0^{t\land T^b_n} e^{-qs}v_{b^\ast}^{\prime} (U^b_{s-}) \diff R^{b,c}_s \right] \\
& + \bE_x \Big[ \sum_{0\leq s\leq {t\land T^b_n}} e^{-qs}\rbra{ { v_{b^\ast}(U^b_{s-} + \Delta X_s + \Delta R_s^b) }
-v_{b^\ast}
(U^b_{s-}+\Delta X_s)  }\Big].
\end{aligned}
\end{align*} 
Because $v_{b^*}$ is of polynomial growth {by Lemma \ref{Rem202}}, we have $\bE_x \sbra{e^{-q{(t\land T_n^b)}} v_{b^\ast}(U^{b}_{t\land T^b_n})} {\xrightarrow{t,n \uparrow \infty}}0$ from the proof of \cite[Theorem 2]{Yam2017}. 
In addition, when $R_s^{b,c}$ increases $U_s^b = b$, and when $\Delta R_s^b > 0$ (i.e. $U_s^b = U_{s-}^b + \Delta X_s + \Delta R_s^b \neq U_{s-}^b + \Delta X_s$) we have $U_s^b =U_{s-}^b + \Delta X_s + \Delta R_s^b = b$, and hence
\begin{multline*}
\bE_x\Big[\int_0^{t\land T^b_n} e^{-qs}v_{b^\ast}^\prime  (U^b_{s-})\diff R^{b,c}_s\Big] + \bE_x \Big[ \sum_{0\leq s<t\land T^b_n} e^{-qs}\rbra{{ v_{b^\ast}(U^b_{s-} + \Delta X_s + \Delta R_s^b) }-v_{b^\ast}(U^b_{s-} + \Delta X_s)}\Big] \\
=\bE_x\Big[\int_0^{t\land T^b_n} e^{-qs}v_{b^\ast}^\prime  (b)\diff R^{b,c}_s\Big] + \bE_x \Big[ \sum_{0\leq s<t\land T^b_n} e^{-qs}\rbra{v_{b^\ast}(b)-v_{b^\ast}({b - \Delta R_s^b})}\Big] \\ =- C \bE_x\Big[ \int_{[0,t\land T^b_n]} e^{-qs} \diff R^{b}_s\Big],
\end{multline*}
where the last equality holds because, using the fact that $v_{b^*}^\prime(x) = - C$ for  $x \in (-\infty, b^*)$ and $b - \Delta R_s^b < b < b^*$, 
$v_{b^\ast}(b)-v_{b^\ast}(b - \Delta R_s^b) = - C \Delta R_s^b$.
Hence, {with}
\begin{align}
v_b^{t,n} (x):=\bE_x \left[ \int_0^{t\land T^b_n} e^{-qt}f(U^b_t)\diff t + C\int_{[0, t\land T^b_n]}e^{-qt} \diff R^b_t \right], \quad x\in\bR,   
\end{align}
we have, using Lemma \ref{Lem205}, 
\begin{multline*}
\bE_x \left[ e^{-q{(t\land T^b_n)}}v_{b^\ast}(U^b_{t\land T^b_n})\right]-v_{b^\ast}(x)
=\bE_x \left[ \int_0^{t\land T^b_n}  e^{-qs}(\cL - q)v_{b^\ast}(U^b_{s-})\diff s \right]
- C \bE_x\left[\int_{[0, t\land T^b_n]} e^{-qs} \diff R^{b}_s\right] \\
=\bE_x \left[ \int_0^{t\land T^b_n}  e^{-qs} g(U^b_{s})\diff s \right] - v_b^{t,n}(x)
=\bE_x\left[\int_0^{t\land T^b_n} e^{-qs}g(U^b_{s}) {1_{\{ -\infty < U^b_{s} < b^\ast  \}}} \diff s\right] - v_b^{t,n}(x). 
\end{multline*}
Because $v_{b^*}$ is of polynomial growth {by Lemma \ref{Rem202}}, we have $\bE_x \sbra{e^{-q{(t\land T_n^b)}} v_{b^\ast}(U^{b}_{t\land T^b_n})} {\xrightarrow{t,n \uparrow \infty}}0$ from the proof of \cite[Theorem 2]{Yam2017}. Since $\pi^b$ is admissible {(satisfying \eqref{2})} and using the dominated convergence theorem, we have $v_b^{t,n}(x) {\xrightarrow{t,n \uparrow \infty}}v_b(x)$. 
In addition, the function $g$ is continuous on $\bR$ {and hence finite on a finite interval,}
and thus $g(U^b_{s})1_{(-\infty , b^\ast )}  (U^b_{s})$ is bounded for $s\geq 0$. 
Therefore, using the dominated convergence theorem, we have \eqref{27}.

\par
{We are now ready to complete the proof.}
Because $g$ is convex and $g(b^*) = 0$, if there exists $a < b^*$ such that $g(a) < 0$, then necessarily $g(x) < 0$ for {all} $x \in (a, b^*)$. {Setting $b = a$ in \eqref{27}  and by Theorem \ref{Thm202},
\begin{align}
\bE_x\left[\int_0^\infty e^{-qs}g(U^a_{s}) {1_{\{ -\infty < U^a_{s} < b^\ast  \}}}  \diff s\right] = v_a(x)-v_{b^\ast}(x) \geq 0. \label{g_v_positive}
\end{align}}
On the other hand,  because $U^a_t \geq a$ a.s., $\int_0^\infty e^{-qs}g(U^a_{s}) {1_{\{ -\infty < U^a_{s} < b^\ast  \}}}   \diff s$ is {nonpositive} and also strictly negative with a positive probability when the starting point $x$ is less than $b^\ast$. {This contradicts with \eqref{g_v_positive} and the proof is complete.} 
\end{proof}

\subsection{General optimality result} \label{subsection_optimality_general}

We shall now complete the proof of Theorem \ref{theorem_main_2} by relaxing Assumption \ref{assump_C_line} {for the case $X$ is of unbounded variation}. This will be done via approximation.


For $\varepsilon>0$, we define 
\begin{align}
f^{(\varepsilon)}(x):= 
\begin{cases}
f(b^\ast)+{\int_{b^*}^x} f^{(\varepsilon)\prime}(y) \diff y, \quad& x\geq b^\ast, \\
f(b^\ast)-{\int_x^{b^*}} f^{(\varepsilon)\prime}(y) \diff y, \quad& x<b^\ast, 
\end{cases}
\end{align} 
where the derivative is given as follows:
\begin{align}
f^{(\varepsilon)\prime}(x):= \frac{1}{\varepsilon^2}\int^0_{-\varepsilon} \diff z\int^0_{-\varepsilon} f^\prime_+(x+y+z)\diff y,  \label{f_epsilon_prime}
\end{align}
{which can also be written}
\begin{align}
f^{(\varepsilon)\prime}(x)=\frac{1}{\varepsilon}\int^0_{-\varepsilon} \frac{f(x+z)-f(x+z-\varepsilon)}{\varepsilon}\diff z.\label{b008}
\end{align}

{We confirm that $f^{(\varepsilon)}(x)$ satisfies Assumption \ref{Ass201} (given that $f$ satisfies it). First, the convexity (Assumption \ref{Ass201}(i)) is immediate because the derivative \eqref{f_epsilon_prime} is monotone in $x$.}

Since $f$ is convex, {the left-hand derivative} $f^\prime_-$ is nondecreasing and 
\begin{align}
\frac{f(x+z)-f(x+z-\varepsilon)}{\varepsilon}\in[f^\prime_-(x+z-\varepsilon), f^\prime_-(x+z)], \quad z\in\bR, 
\end{align}
and thus from \eqref{b008}, we have 
\begin{align}
f^{(\varepsilon)\prime}(x)\in[f^\prime_-(x-2\varepsilon), f^\prime_-(x)]\label{b009}
\end{align} 
and, {together with the convexity of $f$ and \eqref{f_epsilon_prime},}
\begin{align} \label{f_prime_epsilon_monotone}
f^{(\varepsilon)\prime}(x) \uparrow f^\prime_{-}(x) \quad \textrm{as } \varepsilon\downarrow0.
\end{align} Therefore, {by monotone convergence}, as $\varepsilon \downarrow0$,
\begin{align}
f^{(\varepsilon)}(x)
\begin{cases}
\uparrow f(x), \quad &x>b^\ast, \\
\downarrow f(x), \quad& x< b^\ast.
\end{cases}
\label{b001}
\end{align}

From \eqref{b009}, we have {the following bounds:}
\begin{align} \label{f_tilde_bound}
f^{(\varepsilon)}(x)\in
\begin{cases}
[f(b^\ast)+f(x-2\varepsilon)-f(b^\ast-2\varepsilon), f(x)], \quad &x>b^\ast, \\
[f(x), f(b^\ast)+f(x-2\varepsilon)-f(b^\ast-2\varepsilon)], \quad &x\leq b^\ast.
\end{cases}
\end{align}
{By this bound and  Assumption \ref{Ass201}(ii) for $f$,  Assumption \ref{Ass201}(ii)  also holds for $f^{(\varepsilon)}$.
Finally, by \eqref{b009} $f^{(\varepsilon)}$ fulfills Assumption \ref{Ass201}(iii). }

{In order to show that the result established in the previous subsection holds for $f^{(\varepsilon)}$, we show the following.}

\begin{lemma} 
For each $\varepsilon > 0$, $f^{(\varepsilon)}$ satisfies Assumption \ref{assump_C_line}.
\end{lemma}
\begin{proof}
Because \eqref{f_epsilon_prime} is increasing in $x$ {by the convexity of $f$}, the function $f^{(\varepsilon)}$ is convex {as well}. Furthermore, $f^{(\varepsilon)} \in C^2(\R)$ with
\begin{align} \label{f_double_prime_epsilon}
f^{(\varepsilon)''}(x) = \frac{1}{\varepsilon^2} (f(x)-2f(x-\varepsilon)+f(x-2\varepsilon)).
\end{align}
 Indeed,
using the dominated convergence theorem {applied to \eqref{b008}}, we have 
 \begin{align}
 \lim_{h\downarrow 0}&\frac{f^{(\varepsilon)\prime}(x+h)-f^{(\varepsilon)\prime}(x)}{h}\\
  &= \lim_{h\downarrow 0}\frac{1}{\varepsilon^2}\int^0_{-\varepsilon} \left(\frac{f(x+h+z)-f(x+z)}{h}-\frac{f(x+h+z-\varepsilon)-f(x+z-\varepsilon)}{h} \right)\diff z
  \\
 &=\frac{1}{\varepsilon^2}\int^0_{-\varepsilon} (f^\prime_+(x+z)-f^\prime_+(x+z-\varepsilon))\diff z
 =\frac{1}{\varepsilon^2} \big(f(x)-2f(x-\varepsilon)+f(x-2\varepsilon) \big),\\
  \lim_{h\downarrow 0}&\frac{f^{(\varepsilon)\prime}(x)-f^{(\varepsilon)\prime}(x-h)}{h}\\
 &=\frac{1}{\varepsilon^2}\int^0_{-\varepsilon} (f^\prime_-(x+z)-f^\prime_-(x+z-\varepsilon))\diff z
 =\frac{1}{\varepsilon^2} \big(f(x)-2f(x-\varepsilon)+f(x-2\varepsilon) \big).
 \end{align}
Finally, by  {\eqref{f_double_prime_epsilon} together with } Assumption \ref{Ass201}(ii), {the second derivative } $f^{(\varepsilon)''}$  is of polynomial growth. 
\end{proof}

By this lemma, Lemma \ref{Lem205a} can be applied when $f$ is replaced with $f^{(\varepsilon)}$. 
Let $\cA^{(\varepsilon)}$ be the set of admissible strategies {when $f$ is replaced with} $f^{(\varepsilon)}$.  We define, for $\pi \in \cA^{(\varepsilon)}$, 
\begin{align}
v^{(\varepsilon)}_\pi (x):= \bE_x \left[\int_0^\infty e^{-qt} f^{(\varepsilon)}(U^\pi_t) \diff t +C \int_{[0, \infty)}e^{-qt}\diff R^\pi_t\right],\quad x\in\bR, 
\end{align}
and
\begin{align}
v^{(\varepsilon)}(x):= \inf_{\pi\in\cA^{(\varepsilon)}} v^{(\varepsilon)}_\pi(x), \quad x\in\bR. \label{v_varepsilon_def}
\end{align}
By Lemma \ref{Lem205a}, the barrier strategy with the barrier 
$b^\ast_{(\varepsilon)}:=\inf \{b\in\bR : \rho^{(\varepsilon)}(b)+C \geq  0 \}$ where
$\rho^{(\varepsilon)}(b) :=\bE_b \left[\int_0^\infty e^{-qt} f^{(\varepsilon)\prime} (U^{b}_t)\diff t\right]$
 is optimal, i.e.,
 \begin{align} \label{v_varepsilon_solved_by_barrier}
 v^{(\varepsilon)}(x) = v^{(\varepsilon)}_{b^*_{(\varepsilon)}}(x), \quad x \in \R,
 \end{align}
{where we define $v^{(\varepsilon)}_b$ analogously to \eqref{def_v_b}.}



\par
We shall now approximate the value function for the original cost function $f$ with the cases with $f^{(\varepsilon)}$ {to show that the original case is solved by a barrier strategy.} 

\begin{lemma} \label{lemma_admissible_set_epsilon}We have
$\cA \subset \cA^{(\varepsilon)}$ for all $\varepsilon>0$.
\end{lemma}
\begin{proof}
Fix $\pi \in \mathcal{A}$ so that, for $x \in \R$,
\begin{align} \label{strategy_set_0}
\bE_x \left[ \int_0^\infty e^{-qt}|f(U^\pi_t)| \diff t \right] &< \infty, \\
\bE_x \left[ \int_{[0, \infty)}e^{-qt} \diff R^\pi_t \right]&< \infty. \label{2_again}
\end{align}




When  $\lim_{y\to\infty}f(y)=\infty$, by \eqref{f_tilde_bound} we also have  $\lim_{y\to\infty}f^{(\varepsilon)}(y)=\infty$ and because $f^{(\varepsilon)}(x) \leq f(x)$ for $x > b^\ast$, we must have  $|f^{(\varepsilon)}(x)| \leq |f(x)|$ for large $x$.

 When $\lim_{y\to\infty}f(y)<\infty$, then necessarily $f$ is nonincreasing by the convexity and by \eqref{f_tilde_bound}, $f^{(\varepsilon)}(x) \geq f(b^\ast)+ f(x-2\varepsilon)-f(b^\ast-2\varepsilon) \geq f(x)+f(b^\ast)-f(b^\ast-2\varepsilon)$ for all $x > {b^*}$ and thus we have $|f^{(\varepsilon)}(x)|\leq |f(x)|+|f(b^\ast)-f(b^\ast-2\varepsilon)|$ for large enough $x$. 
%
These bounds and \eqref{strategy_set_0} show 
\begin{align}
\bE_x \left[ \int_0^\infty e^{-qt}|f^{(\varepsilon)}(U^\pi_t)| 1_{(b^\ast,\infty)} (U^\pi_t) \diff t \right] < \infty. 
\end{align}

%
{On the other hand, } since $X_t\leq U^\pi_t$ for  $t\geq 0$,
\begin{align}
\bE_x& \left[ \int_0^\infty e^{-qt}|f^{(\varepsilon)}(U^\pi_t)|1_{(-\infty, b^\ast)}(U^\pi_t) \diff t\right]\\
&\leq 
\begin{cases}
\bE_x \left[ \int_0^\infty e^{-qt}(|f^{(\varepsilon)}(X_t)|1_{(-\infty, b^\ast)}(X_t)+|\inf_{y\in(-\infty, b^\ast)}f(y)|) \diff t\right], \quad & \textrm{if } \lim_{y\to-\infty}f^{(\varepsilon)}(y)=\infty,\\
\bE_x \left[ \int_0^\infty e^{-qt}(|f^{(\varepsilon)}(X_t)|1_{(-\infty, b^\ast)}(X_t)+|f(b^\ast)|) \diff t\right], \quad & \textrm{if }  \lim_{y\to-\infty}f^{(\varepsilon)}(y)<\infty,
\end{cases}
\end{align}
which is finite by Remark \ref{Rem201} {and \eqref{f_tilde_bound}}. Combining these {and recalling \eqref{2_again}}, we have $\pi \in \mathcal{A}^{(\varepsilon)}$, as desired.
\end{proof}

\begin{lemma} \label{lemma_b_star_epsilon_conv} 
We have
$b_{(\varepsilon)}^\ast \downarrow b^\ast$ 
as $\varepsilon \downarrow 0$. 
\end{lemma}
\begin{proof}
First,  $\varepsilon \mapsto b^*_{(\varepsilon)}$ is decreasing because $\varepsilon \mapsto \rho^{(\varepsilon)}(b)$ is increasing by \eqref{f_prime_epsilon_monotone} for any $b \in \R$.

Because $\rho^{(\varepsilon)}(x)\leq \rho(x)$ {by \eqref{f_prime_epsilon_monotone},}
\begin{align}
b_{(\varepsilon)}^\ast\geq b^\ast. \label{b007}
\end{align}  
We have, for $b\in\bR$, {by \eqref{f_prime_epsilon_monotone},}
\begin{align}
\rho^{(\varepsilon)}(b) \xrightarrow{\varepsilon \downarrow 0} \bE_b\left[\int_0^\infty e^{-qt} f^\prime_- (U^b_t) \diff t\right]
=\rho(b),
\label{a004}
\end{align}
where the equality holds
because 
the potential {measure} of $U^0$ does not have a mass as proved in the proof of  Lemma \ref{Lem402}(i).

We fix $\eta>0$.  By the definition of $b^*$ and by Lemma \ref{Lem402}(ii), we have $\rho(b^\ast+\eta)+C>0$.  
From \eqref{a004}, {
 we can take small $\varepsilon > 0$ such that 
$\rho^{(\varepsilon)}(b^*+\eta)+C
>0$. 
Moreover,  since $b \mapsto \rho^{(\varepsilon)}(b)$ is nondecreasing, 
$b_{(\varepsilon)}^\ast \leq b^\ast + \eta$. }
Since $\eta > 0$ is arbitrary, $\limsup_{\varepsilon \downarrow 0} b_{(\varepsilon)}^\ast  \leq b^\ast$. This together with \eqref{b007} completes the proof.
\end{proof}


\begin{lemma}  \label{lemma_b002}
We have
$v(x) \geq \limsup_{\varepsilon\downarrow0}v^{(\varepsilon)} (x)$ for $x\in\bR$.
\end{lemma}
\begin{proof}
We fix $\bar{\varepsilon}>0$. By \eqref{b001}, we have, for $\varepsilon\in(0,\bar{\varepsilon})$ 
\begin{align}
f^{(\varepsilon)} (x) \in 
\begin{cases}
[f^{(\bar{\varepsilon})}(x), f(x)], \quad &x\geq b^\ast, \\
[f(x), f^{(\bar{\varepsilon})}(x)], \quad &x< b^\ast,
\end{cases}
\end{align}
and thus
\begin{align}
|f^{(\varepsilon)}(x)| \leq |f^{(\bar{\varepsilon})}(x)|+|f(x)|,\quad x\in\bR. \label{b011}
\end{align}
For any $ {\pi} \in\cA(\subset\cA^{(\varepsilon)})$, 
{by Lemma \ref{lemma_admissible_set_epsilon},} 
$\bE_x \left[\int_0^\infty e^{-qt}(|f^{(\bar{\varepsilon})}(U^\pi_t)|+|f(U^\pi_t)|)\diff t\right]<\infty$, 
$x\in\bR$, 
and by \eqref{b011}, we can use the dominated convergence theorem and have 
$\lim_{\varepsilon\downarrow0}\bE_x\left[\int_0^\infty e^{-qt}f^{(\varepsilon)}(U^\pi_t)\diff t\right]
=\bE_x\left[\int_0^\infty e^{-qt}f(U^\pi_t)\diff t\right]$, $x\in\bR$,
and hence
\begin{align}
v_{\pi} (x) =\lim_{\varepsilon\downarrow 0}v_{\pi}^{(\varepsilon)}(x), \quad x \in \bR. \label{b006}
\end{align}

{
Here we show that  $v(x) > -\infty$.
To this end, suppose $v(x) = -\infty$. Then, for any small $-M$, there exists $\tilde{\pi} \in \mathcal{A}$ such that  $v_{\tilde{\pi}}(x) < - M$. Then we have, by \eqref{b006} and then \eqref{v_varepsilon_solved_by_barrier},
$-M > v_{\tilde{\pi}}(x) = \lim_{\varepsilon\downarrow 0}v_{\tilde{\pi}}^{(\varepsilon)}(x) \geq \limsup_{\varepsilon\downarrow 0}v_{b^*_{(\varepsilon)}}^{(\varepsilon)}(x)$.
Because $-M$ is arbitrary this implies $\lim_{\varepsilon \downarrow 0}v_{b^*_{(\varepsilon)}}^{(\varepsilon)}(x) = -\infty$. However, this is impossible. Indeed, because $b^*_{(\varepsilon)} \geq b^*$ by Lemma  \ref{lemma_b_star_epsilon_conv}, the reflected process $U^{b^*_{(\varepsilon)}}$ stays only on $[b^*, \infty)$. This and \eqref{b001} give, for $0 < \varepsilon < \bar{\varepsilon}$, 
\[
\bE_x\left[\int_0^\infty e^{-qt} f^{(\varepsilon)} (U^{b^*_{(\varepsilon)}}_t) \diff t\right] \geq \bE_x\left[\int_0^\infty e^{-qt} f^{(\bar{\varepsilon})} (U^{b^*_{(\varepsilon)}}_t) \diff t\right] \xrightarrow{\varepsilon \downarrow 0}\bE_x\left[\int_0^\infty e^{-qt} f^{(\bar{\varepsilon})} (U^{b^*}_t) \diff t\right] > -\infty
\]
where the convergence holds by monotone convergence and  \eqref{49} and the finiteness holds by Lemma \ref{Thm201}. Hence, we must have $v(x) > -\infty$ by contradiction.
}

Now, for a fixed $\eta >0$ 
and $x\in\bR$, 
there exists $\hat{\pi} \in\cA(\subset\cA^{(\varepsilon)})$ such that 
$v(x)>v_{\hat{\pi}}(x)- \eta$.
From \eqref{b006} {and because $v_{\hat{\pi}}^{(\varepsilon)}(x) \geq v^{(\varepsilon)}(x)$ (see \eqref{v_varepsilon_def})}, we have   
$v(x)>\lim_{\varepsilon\downarrow 0}v_{\hat{\pi}}^{(\varepsilon)}(x)- \eta
\geq \limsup_{\varepsilon\downarrow 0}v^{(\varepsilon)}(x)- \eta$. 
Since {$\eta > 0$ is arbitrary, the proof is complete.}
\end{proof}

\par
{
By Lemma \ref{lemma_b_star_epsilon_conv}  and the continuity of $b \mapsto v_b(x)$ as in Lemma \ref{Lem202}(i), we have 
\begin{align}
\lim_{\varepsilon \downarrow 0}|v_{b^\ast}(x)-v_{b^\ast_{(\varepsilon)}}(x)| = 0. \label{v_diff_epsilon}
\end{align}

Because
$f_+'(x) - f^{(\varepsilon)\prime}(x):= \frac{1}{\varepsilon^2}\int^0_{-\varepsilon} \diff z\int^0_{-\varepsilon} (f_+'(x)- f^\prime_+(x+y+z)) \diff y \geq 0$, $x \in \R$,
the mapping $x \mapsto f(x) - f^{(\varepsilon)}(x)$ is increasing {on $[b^*,\infty)$}. In addition, {as $\varepsilon$ decreases, }$U^{b^*_{(\varepsilon)}}_t$ decreases for each $t \geq 0$ because $b^*_{(\varepsilon)}$ decreases by Lemma  \ref{lemma_b_star_epsilon_conv}. Since $b^*_{(\varepsilon)} \geq b^*$ by Lemma  \ref{lemma_b_star_epsilon_conv}, the reflected process $U^{b^*_{(\varepsilon)}}$ stays only on $[b^*, \infty)$. Hence, for any fixed $\bar{\varepsilon} > 0$, we have  $0 \leq (f - f^{(\varepsilon)}) (U^{b^*_{(\varepsilon)}}_t) \leq (f - f^{(\varepsilon)}) (U^{b^*_{(\bar{\varepsilon})}}_t) $ for all $0 < \varepsilon \leq \bar{\varepsilon}$. By these and monotone convergence,
\begin{align*}
0 \leq v_{b^\ast_{(\varepsilon)}}(x)-v_{b^\ast_{(\varepsilon)}}^{(\varepsilon)}(x) = \bE_x\left[\int_0^\infty e^{-qt} (f - f^{(\varepsilon)}) (U^{b^*_{(\varepsilon)}}_t) \diff t\right]  \leq \bE_x \left[\int_0^\infty e^{-qt} (f - f^{(\varepsilon)}) (U^{b^*_{(\bar{\varepsilon})}}_t) \diff t\right] \xrightarrow{\varepsilon \downarrow 0} 0.
\end{align*}
By this and \eqref{v_diff_epsilon} together with  \eqref{v_varepsilon_solved_by_barrier}, we have
\begin{align}
\lim_{\varepsilon \downarrow 0}|v_{b^\ast}(x)-v^{(\varepsilon)}(x)|= \lim_{\varepsilon \downarrow 0}|v_{b^\ast}(x)-v_{b^\ast_{(\varepsilon)}}^{(\varepsilon)}(x)|  = 0. {\label{c001}} 
\end{align}

{
Now, by Lemma \ref{lemma_b002} and \eqref{c001},
$0 \leq v_{b^*}(x) - v(x) = \limsup_{\varepsilon \downarrow 0}[(v^{(\varepsilon)}(x)- v(x)) {-} (v^{(\varepsilon)}(x)- v_{b^*}(x) )] \leq 0$,
showing $ v_{b^*}(x) = v(x)$ and hence the barrier strategy with barrier $b^*$ is optimal, as desired.}

\section{Driftless compound Poisson cases} \label{sec_compound_poisson_case} 

We now relax Assumption \ref{Ass101}(i) and show that  Theorem \ref{theorem_main_2} holds true when $X$ is a driftless compound Poisson process, i.e. 
$\Psi (\lambda) = \int_{\bR \backslash \{0\} } (1-e^{i\lambda z}) \Pi(\diff z)$ for $\lambda\in\bR$ with $\Pi(\bR \backslash \{0\})< \infty$. We continue to use the same notations/symbols for a compound Poisson process $X$. 
%

For $\varepsilon\in\bR$, we define 
\begin{align}
X^{[\varepsilon]}_t = X_t +\varepsilon t, \quad t\geq 0. 
\end{align}
Accordingly, for $b\in\bR$, let $R^{[\varepsilon], b}=\{R^{[\varepsilon], b}_t : t\geq 0\}$, $U^{[\varepsilon], b}=\{U^{[\varepsilon], b}_t : t\geq 0\}$ and $v_{b}^{[\varepsilon]}$, respectively, be  the cumulative amount of control, the resulting {controlled} process and the expected cost when we apply the barrier strategy at $b$ to $X^{[\varepsilon]}$. 
Let $b^\ast_{[\varepsilon]}$ be the barrier defined by \eqref{def_b_star} for $X^{[\varepsilon]}$.

By Theorem \ref{theorem_main_2}, the barrier strategy at $b^\ast_{[\varepsilon]}$ is optimal for the problem driven by  $X^{[\varepsilon]}$ {at least} when $\varepsilon \neq 0$.  Our objective in this section is to show that this remains to hold  when $\varepsilon = 0$, i.e., $v(x)=v_{b^\ast}(x)$ for all $x \in \R$, where the barrier $b^* = b^*_{[0]}$ is defined in  
\eqref{def_b_star}, which clearly makes sense even for a driftless compound Poisson process $X = X^{[0]}$. In the remaining, we can safely  drop the subscript/superscript $[\varepsilon]$ when $\varepsilon = 0$. Also, we fix the initial value $x \in \R$ for the rest of this section.
%
%
%

{First, for } $\varepsilon_2 > \varepsilon_1$, we have {the following bounds:}
\begin{align} \label{bound_wrt_epsilon}
U^{[\varepsilon_2], b}_t- U^{[\varepsilon_1], b}_t \in [0, t(\varepsilon_2 - \varepsilon_1)], 
\quad 
R^{[\varepsilon_1], b}_t- R^{[\varepsilon_2], b}_t \in [0, t(\varepsilon_2 - \varepsilon_1)],
\quad t\geq 0.  
 \end{align}
To see the above, the latter holds because 
 $X_t^{[\varepsilon_2]}-X_t^{[\varepsilon_1]} = t (\varepsilon_2-\varepsilon_1)$ and by the definition of $R^{[\varepsilon], b}$ in terms of the running infimum of $X^{[\varepsilon]}$, which directly implies the former.
%
%
%
\begin{lemma} \label{lemma_left_b}
We have ${\varepsilon \mapsto} b^*_{[\varepsilon]}$ is nonincreasing {on $\R$} and left-continuous at $0$.
\end{lemma}
\begin{proof}
By the monotonicity of $\varepsilon \mapsto U^{[\varepsilon], b}_t$ {(shown by the former of  \eqref{bound_wrt_epsilon})} and the convexity of $f$ as in Assumption \ref{Ass201}(i), the mapping $\varepsilon \mapsto \rho_{[\varepsilon]}(b)
 := \bE_b \left[\int_0^\infty e^{-qt} f^\prime_+ (U^{[\varepsilon], b}_t)\diff t\right] {=\bE \left[\int_0^\infty e^{-qt} f^\prime_+ (U^{[\varepsilon], 0}_t+b)\diff t\right]}$ 
is nondecreasing, and thus $\varepsilon \mapsto b^*_{[\varepsilon]}$ is nonincreasing.  Therefore, $b^*_{[-0]} := \lim_{\varepsilon \downarrow 0} b^*_{[-\varepsilon]} {\geq b^*}$ exists.
{For the left-continuity, it suffices to show} 
 $b^*= b^*_{[-0]}$. {For the case $X$ is the negative of a subordinator, then, for $\varepsilon > 0$, $U^{[-\varepsilon],b} \equiv U^{b} = b$ uniformly in time $\mathbb{P}_b$-a.s.\ for all $b \in \mathbb{R}$ and hence $b^* = b^*_{[-\varepsilon]}$, implying the left-continuity. Hence, below we assume $X$ is not the negative of a subordinator.}

We suppose $b^*< b^*_{[-0]}$ to derive a contradiction. By this assumption, we can take $\tilde{b} \in (b^*, b^*_{[-0]})$. Again, by the monotonicity of $\varepsilon \mapsto U^{[\varepsilon], b}_t$ and the convexity of $f$,
 monotone convergence gives $\rho_{[- \varepsilon]}(\tilde{b}) \xrightarrow{\varepsilon \downarrow 0} \bE \left[\int_0^\infty e^{-qt} f^\prime_- (U^{0}_t + \tilde{b})\diff t\right] > \bE \left[\int_0^\infty e^{-qt} f^\prime_+ (U^{0}_t + b^*)\diff t\right] = \rho(b^*) \geq -C$, {where the strict inequality holds by Lemma \ref{Lem402}(ii).} This contradicts with the fact that $\rho_{[-\varepsilon]}(\tilde{b}) < -C$ for all $\varepsilon > 0$ (implied by the assumption that $\tilde{b} < b^*_{[-0]}$), {completing the proof.}
%
\end{proof}


{By Lemma \ref{lemma_left_b}, we have 
\begin{align}
\beta(\varepsilon) &:= b^\ast_{[-\varepsilon]}-b^\ast \geq 0, \quad \varepsilon > 0, \label{beta_epsilon} \\
\beta(\varepsilon) &\xrightarrow{\varepsilon \downarrow 0} 0. \label{beta_vanish}
\end{align}
}
{For all $\varepsilon > 0$ and $t \geq 0$,} {by applying
\eqref{6}-\eqref{8}} and \eqref{bound_wrt_epsilon} to
$U^{[-\varepsilon], b^\ast_{[-\varepsilon]}}_t-U^{b^\ast}_t=
( U^{ b^\ast_{[-\varepsilon]}}_t-U^{b^\ast}_t)+(U^{[-\varepsilon], b^\ast_{[-\varepsilon]}}_t - U^{ b^\ast_{[-\varepsilon]}}_t)$ and
$R^{[-\varepsilon], b^\ast_{[-\varepsilon]}}_t-R^{b^\ast}_t=
( R^{ b^\ast_{[-\varepsilon]}}_t-R^{b^\ast}_t)+(R^{[-\varepsilon], b^\ast_{[-\varepsilon]}}_t-R^{ b^\ast_{[-\varepsilon]}}_t )$, we have 
\begin{align}
U^{[-\varepsilon], b^\ast_{[-\varepsilon]}}_t-U^{b^\ast}_t \in  [ -t\varepsilon, {\beta(\varepsilon)} ], 
\quad
R^{[-\varepsilon], b^\ast_{[-\varepsilon]}}_t-R^{b^\ast}_t \in  [0, t\varepsilon+ {\beta(\varepsilon)} ]. \label{103}
\end{align}
{By the former and} the convexity of $f$, for all $t \geq 0$,
\begin{align*}
| f(U^{[-\varepsilon], b^\ast_{[-\varepsilon]}}_t)-f(U^{b^\ast}_t) | &\leq  | U^{[-\varepsilon], b^\ast_{[-\varepsilon]}}_t -U^{ b^\ast}_t |
\Big(|f^\prime_+(U^{[-\varepsilon], b^\ast_{[-\varepsilon]}}_t   )|\lor|f^\prime_+   (U^{b^\ast}_t)| \Big)
\\ &\leq (t\varepsilon+{\beta(\varepsilon)})
\Big(|f^\prime_+(U^{[-\varepsilon], b^\ast_{[-\varepsilon]}}_t   )|\lor|f^\prime_+   (U^{b^\ast}_t)| \Big).
\end{align*}
By this, integration by parts and \eqref{103},
for $\varepsilon > 0$, we have 
\begin{align}
&\left| v_{b^\ast_{[-\varepsilon]}}^{[-\varepsilon]}(x)-v_{b^\ast}(x) \right|\\
&\leq \bE_x \left[\int_0^\infty e^{-qt} | f(U^{[-\varepsilon], b^\ast_{[-\varepsilon]}}_t)-f(U^{b^\ast}_t) |\diff t \right]
 +|C|q \bE_x \left[\int_0^\infty e^{-qt}(R^{[-\varepsilon], b^\ast_{[-\varepsilon]}}_t-R^{b^\ast}_t)\diff t \right] \\
&\leq
\bE_x \left[\int_0^\infty e^{-qt}(t\varepsilon+{\beta(\varepsilon)})
\Big(|f^\prime_+(U^{[-\varepsilon], b^\ast_{[-\varepsilon]}}_t   )|\lor|f^\prime_+   (U^{b^\ast}_t)| \Big)
\diff t \right] +|C|q \bE_x \left[\int_0^\infty e^{-qt}(t\varepsilon+ {\beta(\varepsilon)} )\diff t \right]. 
\end{align}


Fix $\overline{\varepsilon} > 0$ and take $q_1\in (0, q)$ and $k_{q_1}>0$ such that $te^{-qt} < k_{q_1} e^{-q_1 t}$ for all $t\geq 0$. For $0 < \varepsilon < \overline{\varepsilon}$, because (recalling $b^\ast \leq b^\ast_{[-{\varepsilon}]} \leq b^\ast_{[-\overline{\varepsilon}]}$ by \eqref{beta_epsilon}) 
$U^{[-\overline{\varepsilon}], b^\ast}_t \leq U^{[-\varepsilon], b^\ast_{[-\varepsilon]}}_t \leq U^{b^\ast_{[-\overline{\varepsilon}]}}_t$ and $U^{[-\overline{\varepsilon}], b^\ast}_t \leq U^{b^\ast}_t\leq U^{b^\ast_{[-\overline{\varepsilon}]}}_t$ {and $f$ is convex}, we have  
\begin{align}
|f^\prime_+(U^{-[\varepsilon], b^\ast_{[-\varepsilon]}}_t   )|\lor|f^\prime_+   (U^{b^\ast}_t)|
\leq |f^\prime_+(U^{[-\overline{\varepsilon}], b^\ast}_t   )|+|f^\prime_+   (U^{b^\ast_{[-\overline{\varepsilon}]}}_t)|.
\end{align}
{Therefore,}
 \begin{align*}
 \left| v_{b^\ast_{[-\varepsilon]}}^{[-\varepsilon]}(x)-v_{b^\ast}(x) \right|
 \leq  
 \bE_x \left[\int_0^\infty \left(\varepsilon k_{q_1} e^{-q_1 t} + e^{-qt} {\beta(\varepsilon)}\right)
\Big(|f^\prime_+(U^{[-\overline{\varepsilon}], b^\ast}_t   )|+|f^\prime_+   (U^{b^\ast_{[-\overline{\varepsilon}]}}_t)|+|C|q\Big)
\diff t \right]
\xrightarrow{\varepsilon \downarrow 0}0, 
 \end{align*}
{where the integrability of the expectation 
 can be shown as in} the proof of Lemma \ref{Lem201} and \eqref{beta_vanish}.
 Hence, recalling that the optimal value function when driven by $X^{[-\varepsilon]}$ is $v^{[-\varepsilon]}(x)=  v^{[-\varepsilon]}_{b^\ast_{[-\varepsilon]}}$ when $\varepsilon \neq 0$,
 \begin{align}
 \lim_{\varepsilon\downarrow0} v^{[-\varepsilon]}(x)=  \lim_{\varepsilon\downarrow0} v^{[-\varepsilon]}_{b^\ast_{[-\varepsilon]}}(x)=v_{b^\ast}(x) \geq v(x). \label{104}
 \end{align}
 
%

For $\delta>0$, we can take {an admissible strategy (for the problem driven by the original process $X$)} $\pi_{(\delta)}\in \cA $ 
such that $v(x)+\delta>v_{\pi_{(\delta)}} (x)$.   For $\varepsilon > 0$, we define $\pi^{[-\varepsilon]}_{(\delta)}=\{ R^{\pi_{(\delta)}}_t + \varepsilon t: t\geq 0\}$,
 which is also admissible for the problem driven by $X^{[-\varepsilon]}$ (since $\pi^{[-\varepsilon]}_{(\delta)}$ is adapted to the filtration generated by $X^{[-\varepsilon]}$). 
{Let $U^{[-\varepsilon], \pi^{[-\varepsilon]}_{(\delta)}}_t$ be the controlled process in \eqref{def_U_pi} and $v^{[-\varepsilon]}_{\pi^{[-\varepsilon]}_{(\delta)}}(x)$ its corresponding expected cost  \eqref{v_pi} for the problem driven by $X^{[-\varepsilon]}$.}
Because
$U^{\pi_{(\delta)}}_t = X_t +  R^{\pi_{(\delta)}}_t =  X_t^{[-\varepsilon]} + (R^{\pi_{(\delta)}}_t + \varepsilon t) = U^{[-\varepsilon], \pi^{[-\varepsilon]}_{(\delta)}}_t$ for $t \geq 0$,
\begin{align}
\Big|v^{[-\varepsilon]}_{\pi^{[-\varepsilon]}_{(\delta)}}(x)-v_{\pi_{(\delta)}}(x) \Big|
\leq |C|\bE_x \left[\int_0^\infty e^{-qt}  \diff (R^{\pi^{[-\varepsilon]}_{(\delta)}}_t -R^{\pi_{(\delta)}}_t) \right]
=|C|\varepsilon \int_0^\infty e^{-qt} \diff t \xrightarrow{\varepsilon \downarrow 0} 0.
\end{align}
Hence,
$v(x)+\delta > v_{\pi_{(\delta)}} (x) = \lim_{\varepsilon\downarrow0}v^{[-\varepsilon]}_{\pi^{[-\varepsilon]}_{(\delta)}}(x)
\geq \lim_{\varepsilon\downarrow0} v^{[-\varepsilon ]}(x)$, 
{where  the last inequality holds because $v^{[-\varepsilon]}$ is the optimal value function (for the problem driven by $X^{[-\varepsilon]}$).}
{Because $\delta > 0$ is arbitrary,
$v(x)\geq \lim_{\varepsilon\downarrow0} v^{[-\varepsilon]}(x)$, which, together with \eqref{104}, shows $v(x)=v_{b^\ast}(x)$. } 
{This completes the proof  of Theorem \ref{theorem_main_2} when $X$ is a driftless compound Poisson process.}
}
%

\section{Concluding remark} \label{section_conclusion} 

In this paper, we studied a classical singular control problem for \lev processes and showed the optimality of a barrier strategy.  We obtained a concise expression of the optimal strategy that holds for a general class of \lev processes of bounded or unbounded variation. 

There are various venues for future research. First, it is natural to consider the case with fixed costs, where the objective is to obtain an optimal impulse control. In the spectrally negative case by \cite{Yam2017}, the $(s,S)$-policy is shown to be optimal for a suitable selection of barriers $s$ and $S$.  However, the method used in \cite{Yam2017}  applies only to the spectrally negative L\'evy case. The verification involving integro-differential equations is extremely challenging for a general L\'evy process with two-sided jumps. However, the pathwise analysis obtained in this paper is expected to hold similarly and hence the derivatives with respect to the starting point as well as the barriers $s$ and $S$ can be written in a similar way. The biggest challenge is to show the quasi-variational inequality, which requires non-standard techniques as in \cite{Benkherouf}  even in the spectrally negative case.

It is also of great interest to consider the versions with restricted sets of strategies.  In \cite{Hernandez}, the optimality of a refraction strategy was shown for the case the control process is assumed to be absolutely continuous with a bounded density with respect to the Lebesgue measure.  In \cite{Perez_Yamazaki_Bensoussan}, the optimality of a certain barrier strategy was shown for the case when control opportunities arrive only at Poisson arrival times. These results are based on the assumption of spectrally negative \lev processes, but from the conclusions obtained in this paper (see in particular  Remark \ref{remark_rho_variants}), the optimality is conjectured to hold for a general \lev process as well. The pathwise and smoothness analysis and the technique for verification are certainly helpful in tackling these problems.

The connection between singular control and optimal stopping is well known  (see, e.g., \cite{Boetius,KS1,KS2,Oksendal}), and pursuing this is an alternative approach for solving a singular control problem. In this paper, however, we avoided this approach and focused on solving the singular control problem directly for two reasons.  First, the concise characterization of the optimal barrier via \eqref{def_rho} and other expressions in terms of the reflected process are natural results obtained by focusing on the singular control.  Moreover, technical proofs such as those for smoothness rely on analytical properties of reflected processes. For these reasons, reduction to optimal stopping is unlikely to simplify the solution for the considered problem.  Second reason is related to the discussion in Remark \ref{remark_rho_variants} that the optimal barriers in  \cite{Hernandez} and \cite{Perez_Yamazaki_Bensoussan} are expressed in the same way with $U_t^b$ replaced by variants of the reflected process. This generality is likely to be lost when transformed to an optimal stopping problem.  Many of the results obtained in this paper can be directly used when a general \lev case is considered for    \cite{Hernandez} and \cite{Perez_Yamazaki_Bensoussan}.
However, for a different formulation, for example with a terminal horizon, random discounting and more general running cost functions, reduction to optimal stopping problem may become a more efficient approach. Optimal stopping for a \lev process is significantly more challenging than the Brownian motion or diffusion case and existing results are still limited.
However, several fluctuation theory approaches for optimal stopping have been recently developed for a general \lev process (see, e.g., \cite{Long_Zhang, Surya}). 

Finally,  the methods developed in this paper can potentially be applied to other singular control problems.  A majority of recent developments in the \lev model focus on the spectrally one-sided cases.  In de Finetti's optimal dividend problem, for example, it is standard to model the surplus of an insurance company in terms of a spectrally negative \lev process (see, among others, \cite{AvrPalPis2007, Loeffen}). The dual model driven by spectrally positive \lev processes has also been studied actively \cite{Avanzi, BE, BKY}.
Other singular control problems for spectrally negative \lev processes include
\cite{BY, Hernandez_Yamazaki}, which require two barriers to characterize the optimal strategy. 
While spectrally one-sided \lev models typically admit semi-explicit solutions, often written in terms of the scale function, in many applications, it is more realistic to consider processes with jumps in both directions and, in some cases, subordinators.
Our approach works for a general class of  \lev processes including subordinators/driftless  compound Poisson processes, for which classical scale function/Wiener-Hopf theory cannot be directly applied. 

\appendix 
\section{Proofs}

\subsection{{Proof of Lemma \ref{Rem201A}}} \label{sec0A}
{(i)}
By Assumption \ref{Ass101}(i) and \cite[Exercise 6.4]{Kyp2014}, we have 
\begin{align}
\int_0^\infty 1_{\{\sup_{s\in[0,t]} X_s=X_t \}}\diff t= 0 \quad \text{or}\quad \int_0^\infty 1_{\{\inf_{s\in[0,t]} X_s=X_t \}}\diff t= 0,
\quad \text{{$\p$-}a.s.}\label{32}
\end{align}

To show that \eqref{32} implies that $0$ is regular for $\bR \backslash \{0\}$, suppose 
%
to derive a contradiction that $0$ is not regular for $\mathbb{R} \backslash \{0\}$.
Then, Blumenthal's zero-one law gives
$\bP(T_{\bR\backslash\{0\}} > 0) {= 1}$, 
where 
$T_{\bR\backslash\{0\}}:= \inf\{t>0: X_t\neq 0\}$.
Since $1=\bP(T_{\bR\backslash\{0\}}>0)=\lim_{n\uparrow\infty}\bP(T_{\bR\backslash\{0\}}>\frac{1}{n})$ by the dominated convergence theorem,
there exists $\epsilon>0$ such that
$\bP(T_{\bR\backslash\{0\}}>\epsilon)>0$. 
{On} $\{T_{\bR\backslash\{0\}}>\epsilon\}$, we have $X_t=0$ for $t \in [0, \epsilon]$ which implies that 
{$\sup_{s\in[0,t]} X_s=\inf_{s\in[0,t]} X_s=0$} for $t \in [0, \epsilon]$, {contradicting \eqref{32}.}

\par

{(ii)} 
We now show the continuity of $x\mapsto\tau^-_x$.  
For $x>0$, we have $\tau^-_x=0$ {$\p$-a.s.,} and so the continuity is obvious. Thus we assume $x \leq 0$ {for the rest of the proof}. 

From the definition of $\tau^-_x$ {as in \eqref{def_tau_b}}, it is {immediate} that $\bP$-a.s.\ as $\varepsilon\downarrow 0$, 
$\tau^-_{x-\varepsilon}\downarrow \tau^-_x$ and 
$\tau^-_{x+\varepsilon}\uparrow T^-_x\leq\tau^-_x$, 
where $T^-_x:= \inf \{t\geq 0: X_t \leq x\}$. 
This shows the left-continuity for $x  \leq 0$ (including $x = 0$). It now remains to show for $x < 0$ that the left- and right-limits coincide.
Here, we want to prove $\tau^-_x = T^-_x$, $\bP$-a.s. for $x \in (-\infty , 0)$. 
By the strong Markov property, we have 
\begin{align}
\bE \left[ e^{- \tau^-_x}\right]=\bE \left[e^{-  T^-_x}\bE_{X_{T^-_x}} \left[ e^{-  \tau^-_x}\right] {1_{\{T_x^- < \infty \}}}\right]. \label{52}
\end{align}
\par
When $0$ is regular for $(-\infty, 0)$, then, {because  $X_{T^-_x} \leq x$, we have } $\bE_{X_{T^-_x}} \left[ e^{-  \tau^-_x}\right]=1$ a.s.\ {on $\{T^-_x < \infty \}$} and thus
the right hand side of \eqref{52} is equal to $\bE \left[e^{-  T^-_x}\right]$ which, {together with the fact that $\tau_x^- \geq T_x^-$ a.s.}, implies that $\tau^-_x = T^-_x$, $\bP$-a.s. 
\par
When $0$ is irregular for $(-\infty, 0)$, then $X$ has bounded variation paths and has a non-negative drift by \cite[Theorem 6.5]{Kyp2014}. 
{
If the process jumps downward onto $x$ {at $T_x^-$}, then by the irregularity it immediately goes up and  $\tau_x^- > T_x^-$ and  therefore $\inf_{t\in[0, \tau^-_x)}X_t \leq X_{T_x^-} = x$ (meaning $\{ \inf_{t\in[0, \tau^-_x)}X_t > x, X_{T_x^-} = x \}$ is {a $\p$-null set}). In other words, when $\inf_{t\in[0, \tau^-_x)}X_t > x$,
we must have $X_{T^-_x}  < x$ and hence {we must have} $\tau_x^- = T_x^-$. Therefore,  $\{\inf_{t\in[0, \tau^-_x)}X_t > x \} \subset \{ \tau_x^- = T_x^- \} \cup \{ T_x^- = \infty \}$, {implying }
\begin{align}
\mathbb{P} \{\tau^-_x > T^-_x, T_x^- < \infty \} =\mathbb{P}  \{ \tau_x^- \neq  T_x^-, T_x^- < \infty \} \leq \mathbb{P}  \{\inf_{t\in[0, \tau^-_x)}X_t \leq x\} = \mathbb{P}  \{\inf_{t\in[0, \tau^-_x)}X_t =x\}.
\end{align}}
Hence, the following lemma completes the proof for the case  $0$ is irregular for $(-\infty, 0)$.
%
%
%
%

\begin{lemma}
If $0$ is {irregular} for $(-\infty, 0)$, then $\bP (\inf_{t\in[0, \tau^-_x)}X_t =x )=0$ for $x\in(-\infty , 0)$. 
\end{lemma}
\begin{proof}
We recall some properties of {the} ladder height processes {(see, e.g.,  \cite[Section 6]{Kyp2014})}. 
Let $H=\{H_t : t\geq 0\}$ and $\hat{H}=\{{\hat{H}}_t : t\geq 0\}$, respectively, be ascending and descending ladder height processes of $X$. {Then, the processes $H$ {and} $\hat{H}$ are subordinators, possibly killed and sent to the cemetery state $+\infty$ at some independent exponential random variable. }
Below, let  $\p^H$ and  $\p^{\hat{H}}$ be the laws of $H$ and $\hat{H}$ when they start at zero.

Since $0$ is regular for $(0, \infty)$ for $X$, 
 it is easy to check that $0$ is regular for $(0, \infty)$ for $H$ {as well}. 
Thus by \cite[Theorem 5.4]{Kyp2014} the potential measure 
$U_H(\diff y)=\bE^H \sbra{\int_0^\infty 1_{\{{H_t} \in \diff y \}} \diff t}$ 
has no atoms {on $(0, \infty)$}. 

{Let $\Pi_{\hat{H}}$ be the \lev measure of $\hat{H}$.}
Since $0$ is irregular for $(-\infty , 0)$ for $X$ and by the definition of the descending ladder height processes, $\Pi_{\hat{H}}$ is a finite measure and $\hat{H}$ has no drift 
({see} \cite[Theorem 6.6 and Section 6.2]{Kyp2014}).
In addition, the measure $\Pi_{\hat{H}}$ has no atoms. 
Indeed, by \cite[Theorem 7.8]{Kyp2014} and since $\hat{\Pi}$ {(denoting the \lev measure of the dual process $-X$)}
has atoms at most countable {points} and 
$U_H$ has no atoms, for $y>0$,
we have by the dominated convergence theorem, for some $k > 0$,
\begin{align*}
\Pi_{\hat{H}}(\{ y \})=
&\lim_{\epsilon \downarrow 0}\Pi_{\hat{H}}(y - \epsilon, \infty)-\Pi_{\hat{H}}(y, \infty)
=\lim_{\epsilon \downarrow 0} k\int_{[0, \infty)} \big(\hat{\Pi} (z+y - \epsilon, \infty)-\hat{\Pi} (z+y, \infty) \big)U_H(\diff z)\\
=&\lim_{\epsilon \downarrow 0} k\int_{[0, \infty)}\hat{\Pi} (z+y - \epsilon, z+y] U_H (\diff z)
= k\int_{[0, \infty)}\hat{\Pi} (\{ z+y\}) U_H(\diff z)=0. 
\end{align*}

Since $0$ is irregular for $(-\infty , 0)$ for $X$ and 
by the definition of the ladder height process,  with $\tau^+_{-x,\hat{H}} :=\inf\{t>0 :\hat{H}_t > -x\}$, we have
$\bP (\inf_{t\in[0, \tau^-_x) }X_t =x )
=\bP^{\hat{H}} \Big( \sup_{t\in [0, \tau^+_{-x,\hat{H}}) }{\hat{H}}_{t} ={-x} \Big)
=\bP^{\hat{H}}\rbra{{\hat{H}}_{\tau^+_{-x,\hat{H}}-} ={-x}}$.
By the compensation formula of the Poisson point processes and since $\hat{H}$ has no drift, we have, {with $\cN_{\hat{H}}$ the Poisson random measure 
on $([0, \infty) \times\bR ,\cB [0, \infty ) \times (0, \infty), \diff s \times \Pi_{\hat{H}}(\diff x)) $ 
associated with the jumps of $\hat{H}$, } 
\begin{multline}
{\bP^{\hat{H}}\rbra{{\hat{H}}_{\tau^+_{-x,\hat{H}}-} ={-x}}}
=\bE^{\hat{H}} \sbra{\int_{[0, \infty) \times (0, \infty)}
1_{\{ {\hat{H}}_{t-}=-x  \}}1_{\{ {\hat{H}}_{t-}+y >-x  \}}  \cN_{\hat{H}}(\diff t \times \diff y)}\\
=\bE^{\hat{H}} \sbra{\int_0^\infty \diff t 
\int_{(0, \infty)} 
 1_{\{ {\hat{H}}_{t}=- x   \}}1_{\{ {\hat{H}}_{t}+y >-x  \}}  \Pi_{\hat{H}}(\diff y) } = \bE^{\hat{H}} \sbra{\int_0^\infty 
 1_{\{ {\hat{H}}_{t}=-x  \}}   \Pi_{\hat{H}}(-x-{\hat{H}}_{t}, \infty) \diff t }. \label{621}
\end{multline}
Since $\Pi_{\hat{H}}$ is a finite measure and has no atoms, and by \cite[Theorem 5.4]{Kyp2014}, the above is equal to $0$, {as desired.} 
\end{proof}

{(iii)} The proof of (iii) comes from the identity $\bE_x\big[e^{-q\tau^-_0}\big]=\bE \big[e^{-q\tau^-_{-x}}\big]$, (ii) and the dominated convergence theorem.

\subsection{The proof of Lemma \ref{Thm201}} \label{Sec0A1}
{(i)} We first prove 
\begin{align}
\bE_x \left[ \int_0^\infty e^{-qt}\absol{f(U^b_t)}\diff t \right]<\infty, \quad {x \in \mathbb{R}. } \label{3}
\end{align}
Without loss of generality, we assume $b=0$. {Because the strong Markov property gives}
\begin{align}
\bE_x  \left[ \int_0^\infty e^{-qt}\absol{f(U^0_t)}\diff t \right]
\leq \bE_x  \left[ \int_0^\infty e^{-qt}\absol{f(X_t)}\diff t \right]+\bE_x \left[ e^{-q \tau^-_0}\right]
\bE  \left[ \int_0^\infty e^{-qt}\absol{f(U^0_t)}\diff t \right],
\end{align}
and by Remark \ref{Rem201}, it suffices to verify only $\bE  \left[ \int_0^\infty e^{-qt}\absol{f(U^0_t)}\diff t \right]< \infty$. 
Let $T^{(0)}_{-} := 0$ and define recursively, {for $n \geq 1$,} 
$T^{(n)}_+=\inf\{t>T^{(n-1)}_- : U^0_t >1\}$ and $T^{(n)}_-=\inf\{t>T^{(n)}_+ : U^0_t =0\}$.
Using the strong Markov property, we have 
$\bE \left[ \int_0^\infty e^{-qt}\absol{f(U^0_t)}\diff t \right] 
= 
A
+ \sum_{n\in\bN} B_n,$
where
\begin{align*}
A :=  \bE \left[ \sum_{n\in\bN} \int_{T^{(n-1)}_-}^{T^{(n)}_+} e^{-qt}\absol{f(U^0_t)}\diff t \right] \quad \textrm{and} \quad
 B_n:= \bE \left[ e^{-qT^{(n)}_+} \bE_{U^0_{T^{(n)}_+} }\left[  \int_0^{\tau^-_0} e^{-qt}\absol{f(X_t)}\diff t \right]\right].
\end{align*}
Here, we have
\begin{align*}
A \leq \bE\left[ \int_0^\infty e^{-qt}|f(U^0_t)| 1_{\{U^0_t\in [0, 1]\}} \diff t \right]
\leq \bE\left[ \int_0^\infty e^{-qt} \sup_{0 \leq y \leq 1} |f(y)| \diff t \right]=\frac{1}{q}\sup_{0 \leq y \leq 1} |f(y)|.
\end{align*}
On the other hand, {for $n \geq 1$,} $U_{T_-^{(n-1)}}^0 = 0$ and hence $(T_+^{(n)} -T_-^{(n-1)}, U^0_{T_+^{(n)}})|_{\mathcal{F}_{T_-^{(n-1)}}} \sim (T_+^{(1)}, U^0_{T_+^{(1)}})$.  Therefore, {with $h(x) := \bE_{x} \left[  \int_0^{\infty} e^{-qt}\absol{f(X_t)}\diff t \right] \geq \bE_{x} \left[  \int_0^{\tau_0^-} e^{-qt}\absol{f(X_t)}\diff t \right]$,} 
\begin{multline*}
B_n \leq \bE \left[  \bE \left[   e^{-qT^{(n)}_+}
{h(U^0_{T^{(n)}_+} )}
 \Big| \mathcal{F}_{T_-^{(n-1)}}\right] \right] = \bE \left[ e^{-qT^{(n-1)}_-}  
 \bE \left[   e^{-qT^{(1)}_+} 
  {h(U^0_{T^{(1)}_+} )}
 \right]\right] \\
  = \bE \left[ e^{-qT^{(n-1)}_-} \right]  \bE \left[   e^{-qT^{(1)}_+} 
   {h(U^0_{T^{(1)}_+} )}
  \right] = {\left(\bE \left[ e^{-qT^{(1)}_-} \right]\right)}^{n-1}  \bE \left[   e^{-qT^{(1)}_+} 
 {h(U^0_{T^{(1)}_+} )}
\right].
\end{multline*}
%
Combining these,
\begin{align}
\bE \left[ \int_0^\infty e^{-qt}\absol{f(U^0_t)}\diff t \right] &\leq \frac{1}{q}\sup_{0 \leq y \leq 1} |f(y)|
+ \sum_{n\in\bN} \left( \bE \left[e^{-qT^{(1)}_{-}}  \right]\right)^{n-1} {\bE \big[ e^{-qT^{(1)}_+} h(U^0_{T^{(1)}_+} )\big]}.
\end{align}

Thus, it suffices to prove that {$\bE \big[ e^{-qT^{(1)}_+} h(U^0_{T^{(1)}_+} ); E_i \big]<\infty$ for $i = 1,2$ where $E_1 := \{ |U^0_{T^{(1)}_+}- U^0_{T^{(1)}_+-} |\leq 1 \}$ and $E_2 := \{ |U^0_{T^{(1)}_+}- U^0_{T^{(1)}_+-} | >  1 \}$.}
By Remark \ref{Rem201}, we have 
$\bE \Big[ e^{-qT^{(1)}_+} h(U^0_{T^{(1)}_+} ); {E_1}
\Big]
{\leq} \sup_{z \in [1, 2]} h(z)<\infty$.  
{By} the compensation {formula} of the Poisson point processes, with {$\bar{U}^0$} the running supremum of $U^0$,
{and $\cN$  the Poisson random measure associated with the jumps of $X$ as in the proof of Proposition \ref{Prop201},} 
\begin{align}
&\bE \left[ e^{-qT^{(1)}_+} h(U^0_{T^{(1)}_+} ); {E_2}
\right]
=\bE \left[ \int_{[0, \infty)\times(1, \infty)} e^{-qt}h(U^0_{t-}+y ) 1_{\{ \bar{U}^0_{t-} \leq 1
\}}\cN(\diff t \times \diff y) \right]\\
=&\bE \left[ \int_{(1, \infty )}\Pi( \diff y)  \int_0^\infty e^{-qt} h(U^0_{t-}+y ) 1_{\{ \bar{U}^0_{t-} \leq 1
\}}\diff t\right] \\
=&\int_{(1, \infty )}\Pi( \diff y) \int_{[0,1]} h(z+y ) 
\bE\left[\int_0^{T^{(1)}_+}e^{-qt} 1_{\{{U^0_{t}}  \in \diff z\}} \diff t\right] 
\leq \frac{1}{q} \int_{(1, \infty )} \Big(\sup_{z\in [0, 1 ]}h(z+y ) \Big)\Pi( \diff y), 
\end{align}
{which} is finite since $h$ is {of} polynomial growth from Remark \ref{Rem201} and by Assumption \ref{Ass101}. {Hence \eqref{3} holds.}

\par
{(ii)}
Fix any arbitrary constant $u>0$. We have 
\begin{align}
\bE_x &\left[\int_{[0, \infty)}e^{-qt} \diff R^0_t \right]
=\sum_{k\in\bN} \bE_x\left[ \int_{[(k-1)u, ku)}e^{-qt} \diff R^0_t \right]\\
\leq& {(-x)} \lor0+\sum_{k\in\bN} \bE_x\left[ -e^{-q(k-1)u}\inf_{s\in [(k-1)u, ku)} (X_s- X_{(k-1)u-}) \right]\\
\leq&(-x)\lor0+\sum_{k\in\bN} e^{-q(k-1)u}{\bE}\left[  -\inf_{s\in [0, u)}X_s \right]
=(-x)\lor0+{\bE}\left[  -\inf_{s\in [0, u)}X_s \right]\frac{1}{1-e^{-qu}}.
\end{align}
From the same argument as \cite[Lemma 3.2]{Nob2019}, this is also finite.

{By (i) and (ii), the proof is complete.}

\subsection{The proof of Lemma \ref{Rem202}}\label{AppA03}
 Since $f$ is convex, either $f$ monotonically decreases to {a nonnegative value}
 or otherwise
  there exists {$\alpha > b$} such that $|f|$ is nondecreasing on $(\alpha, \infty)$. 
{ Thus,  for $x,y\in [b , \infty)$ with $x<y$, we have $|f(x)|\leq M_{b,\alpha}  \lor {|f(y)|}
$ where $M_{b,\alpha} := \sup_{z\in[b, \alpha]}|f(z)|$.}
 
{Now}, for $x>0$, {because $U^{b} \geq U^{b-x}$ (see \eqref{7}) and $R^{b} \geq R^{b-x}$  and hence by integration by parts $\int_{[0, \infty)}e^{-qt} \diff R^{b-x}_t \leq \int_{[0, \infty)}e^{-qt} \diff R^{b}_t$,} 
we have 
\begin{align*}
|v_b (x)| 
&{\leq \bE_x \left[ \int_0^\infty e^{-qt}\absol{f(U^b_t)}\diff t\right] + |C|\bE_x \left[ \int_{[0, \infty)}e^{-qt} \diff R^b_t \right]} \\
&{= \bE \left[ \int_0^\infty e^{-qt}\absol{f(U^{b-x}_t + x)}\diff t\right] + |C|\bE \left[ \int_{[0, \infty)}e^{-qt} \diff R^{b-x}_t \right]} \\
&{\leq \bE \left[ \int_0^\infty e^{-qt}( \absol{f(U^b_t+x)} \vee M_{b,\alpha}) \diff t\right]  + |C|\bE \left[ \int_{[0, \infty)}e^{-qt} \diff R^b_t \right] } \\
&\leq \bE \left[ \int_0^\infty e^{-qt}\absol{f(U^b_t+x)}\diff t\right] +\frac{M_{b,\alpha}}{q} + |C|\bE \left[ \int_{[0, \infty)}e^{-qt} \diff R^b_t \right].
\end{align*}
From Assumption \ref{Ass201}(ii), we have {for some $k_1, k_2$ and $N \in \mathbb{N}$,}
\begin{align}
\bE \left[ \int_0^\infty e^{-qt}\absol{f(U^b_t+x)}\diff t\right]&\leq
\bE \left[ \int_0^\infty e^{-qt}(k_1+k_2\absol{U^b_t+x}^N)\diff t\right]\\
&\leq k_1\bE \left[ \int_0^\infty e^{-qt}\diff t\right]+
k_2\sum_{l=0}^N {N \choose l} x^l \bE \left[ \int_0^\infty e^{-qt}\absol{U^b_t}^{N-l}\diff t\right] ,
\end{align}
which is {of} polynomial growth {because $ \bE \left[ \int_0^\infty e^{-qt}\absol{U^b_t}^{N-l}\diff t\right]$ is finite by  
{Lemma \ref{Thm201}}.

\subsection{The proof of Lemma \ref{Lem201}}\label{Sec0B}
Fix $\varepsilon > 0$. 
Since $f$ is a convex function, we have 
$(f(x)-f(x-\varepsilon)) / \varepsilon\leq {f_+^\prime(x)} \leq (f(x+\varepsilon)-f(x) ) / \varepsilon$, $x \in \mathbb{R}$,
which implies that 
\begin{align}
\bE_x \left[ \int_0^\infty e^{-qt}\absol{{f^\prime_+(U^{b}_t)}}\diff t\right]
\leq \bE_x \left[ \int_{0}^\infty e^{-qt}
\frac{\absol{f({U^b_t} +\varepsilon)}+2\absol{f({U^b_t} )}+\absol{f({U^b_t}  -\varepsilon)}}{\varepsilon}
\diff t\right].
\end{align}
This is finite by \eqref{3}
, as desired. 
\subsection{The proof of Lemma \ref{Lem204}}\label{Sec0D}
Since $f^\prime_+$ is right-continuous {and by \eqref{7}}, the map $b\mapsto \bE_b \left[\int_0^\infty e^{-qt} f^\prime_+(U^{b}_t)\diff t\right]$ is right-continuous. 
In addition, we have $f_-^\prime(x)= \lim_{y\uparrow x}f^\prime_{+} (x)$ and thus $\lim_{b^\prime \uparrow b}\bE_{b^\prime} \left[\int_0^\infty e^{-qt} f^\prime_+(U^{b^\prime}_t)\diff t\right] = \bE_b \left[\int_0^\infty e^{-qt} f^\prime_-(U^{b}_t)\diff t\right]$. In view of \eqref{def_b_star} and because $f'_+$ is nondecreasing,  for $\delta > 0$,
\begin{align*}
\bE_{b^\ast} \left[\int_0^\infty e^{-qt}f^\prime_+(U^{b^\ast }_t -\delta)\diff t\right] = \bE_{b^\ast - \delta} \left[\int_0^\infty e^{-qt}f^\prime_+(U^{b^\ast - \delta}_t)\diff t\right]{\leq} -C\leq \bE_{b^\ast} \left[\int_0^\infty e^{-qt}f^\prime_+(U^{b^\ast}_t)\diff t\right]. 
\end{align*}
Hence, taking $\delta \downarrow 0$, 
$\bE_{b^\ast} \left[\int_0^\infty e^{-qt}f^\prime_-(U^{b^\ast}_t)\diff t\right]\leq -C\leq \bE_{b^\ast} \left[\int_0^\infty e^{-qt}f^\prime_+(U^{b^\ast}_t)\diff t\right]$. 
We define, {if $\bE_{b^\ast} \left[\int_0^\infty e^{-qt}f^\prime_+(U^{b^\ast}_t)\diff t\right] \neq \bE_{b^\ast} \left[\int_0^\infty e^{-qt}f^\prime_-(U^{b^\ast}_t)\diff t \right]$,}
\begin{align}
\varepsilon^{\ast} =  \frac {
\bE_{b^\ast} \left[\int_0^\infty e^{-qt}f^\prime_+(U^{b^\ast}_t)\diff t\right] + C} {
\bE_{b^\ast} \left[\int_0^\infty e^{-qt}f^\prime_+(U^{b^\ast}_t)\diff t\right]- \bE_{b^\ast} \left[\int_0^\infty e^{-qt}f^\prime_-(U^{b^\ast}_t)\diff t\right]}
\end{align}
and set it zero otherwise.
Then, setting $f^\prime_{\varepsilon^\ast}(x):=(1-\varepsilon^\ast) f^\prime_+(x) +\varepsilon^\ast f^\prime_-(x)$, $x \in \R$, we have
\begin{align}
\bE_{b^\ast} \left[\int_0^\infty e^{-qt}f^\prime_{\varepsilon^\ast}(U^{b^\ast}_t)\diff t\right]= -C. \label{24}
\end{align}
By {Lemma \ref{Lem203},} \eqref{24} and the strong Markov property (note that $U_{\tau_{b^*}^-} = b^*$ on $\{ \tau_{b^*}^- < \infty \}$), we have 
\begin{align} \label{35}
\begin{split}
v_{b^*,+}^\prime(x)&=\bE_x \left[ \int_0^{\tau^{-}_{b^*}} e^{-qt} f^\prime_+(X_t)\diff t \right]
-C\bE_x\left[e^{-q\tau^-_{b^*}}\right] \\ &=\bE_x \left[ \int_0^{\tau^{-}_{b^*}} e^{-qt} f^\prime_{\varepsilon^*}(U^{b^\ast}_t)\diff t \right] + \bE_x \left[ \int_{\tau^{-}_{b^*}}^{\infty} e^{-qt} f^\prime_{\varepsilon^\ast}(U^{b^\ast}_t)\diff t \right] = \bE_x \left[ \int_0^{\infty} e^{-qt} f^\prime_{\varepsilon^\ast}(U^{b^\ast}_t)\diff t \right],
\end{split}
\end{align}
where the second equality holds because 
 $\bE_x \left[ \int_0^{\tau^{-}_{b^*}} e^{-qt} f^\prime_+(X_t)\diff t \right] = \bE_x \left[ \int_0^{\tau^{-}_{b^*}} e^{-qt} f^\prime_{\varepsilon^\ast}(X_t)\diff t \right]$ as in the proof of Lemma \ref{Lem203}(ii) and we have $U_t^{b^*} = X_t$ for $t < \tau_{b^*}^-$.
Since $f^\prime_{\varepsilon^\ast}$ is {nondecreasing} and $v_{b^\ast}$ is continuous, $v_{b^\ast}$ is convex.

{To complete the proof, we now confirm} the continuity of $v_{b^\ast, +}^\prime$. 
Since $v_{b^\ast, +}^\prime$ is right derivative and $v_{b^\ast}$ is convex, 
$v_{b^\ast, +}^\prime$ is right-continuous and {thus} it suffices to prove $\lim_{\delta\downarrow0}(v_{b^\ast,+}^{\prime}(x)-v_{b^\ast, +}^{\prime}(x-\delta))=0$ for $x\in\bR$. 
Here we use the same notations as the proof in Lemma \ref{Lem203} {for $b = b^*$}. 
For $x \in \mathbb{R}$ and $\delta >0$,  
\begin{align}
v_{b^\ast, +}^{\prime}(x)-v_{b^\ast,+}^{\prime}(x-\delta)
=&\bE \left[ \int_0^\infty e^{-qt} \left(  f^\prime_{\varepsilon^\ast}(U^{(x),b^*}_t)  -f^\prime_{\varepsilon^\ast}( U^{(x-\delta),b^*}_t)\right) \diff t \right]. \label{29}
\end{align}
By \eqref{14}, \eqref{16} and \eqref{18} with $b$ changed to $b^\ast$, $x$ changed to $x-\delta$ and $x+\varepsilon$ changed to $x$ and since $f^\prime_{\varepsilon^\ast}$ is nondecreasing, {by \eqref{29},}
\begin{multline*}
{0 \leq v_{b^\ast, +}^{\prime}(x)-v_{b^\ast,+}^{\prime}(x-\delta)} \leq\bE \left[ \int_0^{\tau^{(x)}_{b^\ast}} e^{-qt} \left(  f^\prime_{\varepsilon^\ast}(U^{(x),b^*}_t)  -f^\prime_{\varepsilon^\ast}( U^{(x),b^*}_t-\delta)\right) \diff t \right]\\
=\bE_{x}\left[ \int_0^{\tau^-_{b^\ast}} e^{-qt} \left(  f^\prime_{\varepsilon^\ast}(U^{b^\ast}_t)  -f^\prime_{\varepsilon^\ast}( U^{b^\ast}_t-\delta)\right) \diff t \right] 
=\bE_{x}\left[ \int_0^{\tau^-_{b^\ast}} e^{-qt} \left(  f^\prime_{\varepsilon^\ast}(X_t)  -f^\prime_{\varepsilon^\ast}( X_t-\delta)\right) \diff t \right]. 
\end{multline*}
By the convexity of $f$, we have $\lim_{y' \uparrow y}f^\prime_{\varepsilon^\ast}(y') = f'_- (y)$ for $y\in\bR$. Thus,
by the monotone convergence theorem, {with $R^{(q)}(x, \cdot)$ the measure as defined in the proof of Lemma \ref{Lem203}(ii),}
\begin{align}
\lim_{\delta\downarrow0}
\bE_{x}\left[ \int_0^{\tau^-_{b^\ast}}e^{-qt} \left(  f^\prime_{\varepsilon^\ast}(X_t)  -f^\prime_{\varepsilon^\ast}( X_t-\delta)\right) \diff t \right]
&= \bE_{x}\left[ \int_0^{\tau^-_{b^\ast}}e^{-qt} {\lim_{\delta \downarrow 0} } \left( f^\prime_{\varepsilon^\ast}(X_t)  -f^\prime_{\varepsilon^\ast}( X_t-\delta)\right) \diff t \right]\\
&\leq \int_{[b^\ast,\infty)}\left(  f^\prime_{\varepsilon^\ast}(y)  -f^\prime_-( y)\right) R^{(q)}(x, \diff y),
\end{align}
{which is zero
since $f^\prime_{\varepsilon^\ast}$ and $f'_-$ differ only at countable points by the convexity of $f$. }
\par
Therefore, by the convexity of $v_{b^\ast}$, $v_{b^\ast,+}^\prime ({x})=v_{b^\ast}^\prime ({x})$ and $v_{b^\ast}$ belongs to $C^1 (\bR)$.

\subsection{Proof of Lemma \ref{Lem205a}} \label{subsection_proof_Lem205a}
{Fix $x \in \mathbb{R}$.}
By Lemma \ref{Lem204} and because $f$ is differentiable {by assumption},
$v_{b^\ast}^\prime(x)=\bE_x \left[ \int_0^{\infty} e^{-qt} f^\prime(U^{b^\ast}_t)\diff t \right]$. 
For $\varepsilon>0$, we have 
\begin{align*}
 \bE_{x+\varepsilon} \left[ \int_0^{\infty} e^{-qt} f^\prime(U^{b^\ast}_t)\diff t \right] - \bE_x \left[ \int_0^{\infty} e^{-qt} f^\prime(U^{b^\ast}_t)\diff t \right]
=   \bE \left[ \int_0^{\infty} e^{-qt} (f^\prime(U^{(x+\varepsilon),b^*}_t) - f^\prime(U^{(x),b^*}_t) ) \diff t \right], 
\end{align*}
{where $U^{(x),b^*}$ is as defined in the proof of Lemma \ref{Lem203} for $b = b^*$.}
Here, {applying \eqref{18} and then  \eqref{14},}
\begin{align}
& \int_0^\infty e^{-qt} (f'(U^{(x+\varepsilon),b^*}_t)-f'(U^{(x),b^*}_t)) \diff t \\
 &= \int_0^{\tau_{b^\ast}^{(x)}} e^{-qt} (f'(U^{(x),b^*}_t+\varepsilon)-f'(U^{(x),b^*}_t)) \diff t  + \int_{\tau_{b^\ast}^{(x)}}^{\tau_b^{(x+\varepsilon)}} e^{-qt} (f'(U^{(x+\varepsilon),b^*}_t)-f'(U^{(x),b^*}_t)) \diff t. 
\end{align}
\par
We {first} prove
\begin{align}
\bE\left[
\int_{\tau_{b^*}^{(x)}}^{\tau_{b^*}^{(x+\varepsilon)}} e^{-qt} \frac {|f^\prime(U^{(x+\varepsilon),b^*}_t)-f^\prime(U^{(x),b^*}_t)|} \varepsilon \diff t  \right] \xrightarrow{\varepsilon \downarrow 0} 0. \label{51}
\end{align}
For $\varepsilon \in (0, 1)$, {because $U_t^{(x),b^*} \leq U^{(x+\varepsilon),b^*}_t \leq U_t^{(x),b^*} + \varepsilon$ for $\tau_{b^*}^{(x)} \leq t \leq \tau_{b^*}^{(x+\varepsilon)}$ as in \eqref{16} and by the mean value theorem,} 
\begin{multline}
\int_{\tau_{b^*}^{(x)}}^{\tau_{b^*}^{(x+\varepsilon)}} e^{-qt} \frac {|f^\prime(U^{(x+\varepsilon),b^*}_t)-f^\prime(U^{(x),b^*}_t)|} \varepsilon \diff t  \leq 
\int_{\tau_{b^*}^{(x)}}^{\tau_{b^*}^{(x+\varepsilon)}} e^{-qt} \frac{1}{\varepsilon}\left( \int_0^\varepsilon|f^{\prime\prime}(U^{(x),b^*}_t+y)| \diff y \right)\diff t \\
\leq
  \int_{\tau_{b^*}^{(x)}}^{\tau_{b^*}^{(x+\varepsilon)}} e^{-qt} \sup_{0\leq y\leq 1} |f^{\prime\prime}(U^{(x),b^*}_t+y)| \diff t  \leq
 \int_0^\infty e^{-qt} \sup_{0\leq y\leq 1} |f^{\prime\prime}(U^{(x),b^*}_t+y)| \diff t,
\label{30}
\end{multline}
{which is integrable because} $f^{\prime\prime}$ is {of} polynomial growth and by the same argument as the proof of Lemma \ref{Thm201}.
Thus, {by the} dominated convergence theorem and because $\tau_b^{(x + \varepsilon)} \xrightarrow{\varepsilon \downarrow 0} \tau_b^{(x)}$ a.s. from Lemma \ref{Rem201A}(ii),}
\begin{align}
\lim_{\varepsilon\downarrow0}
\bE \left[  \int_{\tau_{b^*}^{(x)}}^{\tau_{b^*}^{(x+\varepsilon)}} e^{-qt} \sup_{0\leq y\leq 1} |f^{\prime\prime}(U^{(x),b^*}_t+y)| \diff t \right]=
\bE \left[  \lim_{\varepsilon\downarrow0}\int_{\tau_{b^*}^{(x)}}^{\tau_{b^*}^{(x+\varepsilon)}} e^{-qt} \sup_{0\leq y\leq 1} |f^{\prime\prime}(U^{(x),b^*}_t+y)| \diff t \right]=0. \label{31}
\end{align}
From {this and }\eqref{30}, we have \eqref{51}. 


\par

{Similarly,} for $\varepsilon\in(0, 1)$ and  $t < \tau_{b^*}^{(x)}$, because $U^{(x+\varepsilon),b^*}_t = U_t^{(x),b^*} + \varepsilon = X_t^{(x)} + \varepsilon$ as in \eqref{14},
{
\begin{align}
\int_0^{\tau_{b^*}^{(x)}}e^{-qt} \frac {|f^\prime(X_t^{(x)}+\varepsilon)-f^\prime(X_t^{(x)})|} \varepsilon \diff t 
\leq \int_0^\infty e^{-qt} \sup_{0\leq y\leq 1} |f^{\prime\prime}( {X_t^{(x)}} +y)| \diff t, 
\end{align}
}
which is integrable,  and so using the dominated convergence theorem, 
\begin{multline*}
\bE\left[
\int_0^{\tau_{b^*}^{(x)}}e^{-qt} \frac {f^\prime(U^{(x+\varepsilon),b^*}_t)-f^\prime(U^{(x),b^*}_t)} \varepsilon \diff t  \right] 
{=\bE\left[\int_0^{\tau_{b^*}^{(x)}}e^{-qt} \frac {f^\prime(X_t^{(x)}+\varepsilon)-f^\prime(X_t^{(x)})} \varepsilon \diff t  \right] }\\
=\bE_x\left[\int_0^{\tau_{b^*}^-}e^{-qt} \frac {f^\prime(X_t+\varepsilon)-f^\prime(X_t)} \varepsilon \diff t  \right] \xrightarrow{\varepsilon \downarrow 0} \bE_x\left[\int_0^{\tau_{b^*}^-}e^{-qt} f^{\prime\prime}(X_t)\diff t  \right] . \label{33}
\end{multline*}
{The left-hand derivative can be derived similarly.}

\par
From the arguments above, we obtain 
$v_{b^\ast}''(x)= \bE_x \left[ \int_0^{\tau^-_{b^\ast }} e^{-qt} f^{\prime\prime}
(X_t)\diff t \right]$ for $x\in\R$. 
This is continuous by the dominated convergence theorem. {Here, note that the continuity holds even for $x=b^\ast$, because $\lim_{x\downarrow0}\tau^-_x= 0 =\tau^-_0 $ $\bP$-a.s. since $X$ has unbounded variation paths. }  


\section*{Acknowledgments.  } 
The authors thank the anonymous referees and the associate and area editors for careful reading of the paper and constructive comments and suggestions. 
K. Noba was supported by JSPS KAKENHI grant no. 18J12680 and 21K13807. K. Yamazaki was in part supported by MEXT KAKENHI grant no.\ 19H01791 and 20K03758 and the start-up grant by the School of Mathematics and Physics of the University of Queensland. Both authors were supported by JSPS Open Partnership Joint Research Projects grant no. JPJSBP120209921. 




\end{document}